\documentclass[reqno]{amsart}
\usepackage{amsmath, amsfonts, amsthm, amssymb, setspace, textcomp, bbm, multirow}
\usepackage{geometry}
\newtheorem{theorem}{Theorem}[section]

\newtheorem{corollary}[theorem]{Corollary}
\newtheorem{proposition}[theorem]{Proposition}

\theoremstyle{definition}

\theoremstyle{remark}

\numberwithin{equation}{section}

\newcommand{\Parans}[1]{\left(#1\right)}

\newcommand{\SBrackets}[1]{\left[#1\right]}

\newcommand{\PieceTwo}[4]
{
	\left\{
   	\begin{array}{ll}
      	#1 & #3 \\
       	#2 & #4
     	\end{array}
	\right.
}
\newcommand{\aqprod}[3]{\Parans{#1;#2}_{#3}}
\newcommand{\jacprod}[2]{\SBrackets{#1;#2}_{\infty}}
\newcommand{\Jac}[2]{\left(\frac{#1}{#2}\right)}

\newcommand{\TwoTwoMatrix}[4]
{
	\left(\begin{array}{cc}
		#1 & #2\\#3 & #4
	\end{array}\right)
}

\newcommand{\Floor}[1]{\lfloor #1 \rfloor}
\newcommand{\Fractional}[1]{\left\{#1\right\} }

\newcommand{\z}[2]{\zeta_{#1}^{#2}}
\newcommand{\tmu}{\tilde{\mu}}
\newcommand{\SLTwo}{\mbox{SL}_2(\mathbb{Z})}

\author{CHRIS JENNINGS-SHAFFER}
\address{Department of Mathematics, Oregon State University\\
Corvallis, Oregon 97331, USA
\endgraf cjenningsshaffer@ufl.edu}



\title{Overpartition Rank Differences Modulo 7 By Maass Forms}

\allowdisplaybreaks
\begin{document}

\allowdisplaybreaks

\begin{abstract}
Using that the overpartition rank function is the holomorphic part of a 
harmonic Maass form, we deduce formulas for the rank differences modulo 7.
To do so we make improvements on the current state of the overpartition
rank function in terms of harmonic Maass forms by giving simple
formulas for the transformations under $\SLTwo$ as well as formulas
for orders at cusps.
\end{abstract}

\maketitle

\section{Introduction}
\allowdisplaybreaks

A partition of an integer $n$ is a non-increasing sequence of positive integers 
that sum to $n$. For example the partitions of $4$ are
$4$, $3+1$, $2+2$, $2+1+1$, and $1+1+1+1$. An overpartition of $n$ is a
partition of $n$ in which the first appearance of a part may be overlined. For
example the overpartitions of $4$ are
$4$, $\overline{4}$, 
$3+1$, $3+\overline{1}$, $\overline{3}+1$, $\overline{3}+\overline{1}$,
$2+2$, $\overline{2}+2$, 
$2+1+1$, $2+\overline{1}+1$, $\overline{2}+1+1$, $\overline{2}+\overline{1}+1$, 
$1+1+1+1$,  and $\overline{1}+1+1+1$.
The Dyson rank of a partition or an overpartition is defined as the largest 
part minus the number of parts, in particular the rank does not depend on
whether or not a part is overlined.
For the partitions $4$, $3+1$, $2+2$, $2+1+1$, and $1+1+1+1$ the
respective ranks are $3$, $1$, $0$, $-1$,  and $-3$.

The first point of interest with the rank is that it explains certain 
congruences for the partition function. With $p(n)$ denoting
the number of partitions of $n$, we have $p(5n+4)\equiv 0\pmod{5}$
and $p(7n+5)\equiv 0\pmod{7}$. The rank explains $p(5n+4)\equiv 0\pmod{5}$ in 
that if we group the partitions of $5n+4$ according the value of their rank 
modulo $5$, then we have five groups of equal size. The rank similarly explains
the congruence modulo $7$. While the rank of overpartitions does not so simply
explain congruences for overpartitions, it does yield refinements of congruences
for overpartitions \cite[Theorem 1.2]{BringmannLovejoy}
and it does play a part in explaining 
congruences for the number of appearances of the smallest part in the 
overpartitions of $n$ \cite{GarvanJennings}.

We let $\overline{N}(m,n)$ denote the number of overpartitions of $n$ with
rank $m$ and let $\overline{N}(k,t,n)$ denote the number of overpartitions
of $n$ with rank congruent to $k$ modulo $t$. The generating function for
$\overline{N}(m,n)$ is given by
\begin{align*}
	\mathcal{O}(z,\tau)
	&=
	\sum_{n=0}^\infty \sum_{m=-\infty}^\infty \overline{N}(m,n)z^mq^n
	=
	\sum_{n=0}^\infty
	\frac{\aqprod{-1}{q}{n}q^{\frac{n(n+1)}{2}}}{\aqprod{zq,z^{-1}q}{q}{n}}	
	,
\end{align*}
where here and throughout 
$q=\exp(2\pi i\tau)$ for 
$\tau\in \mathcal{H}$, that is $Im(\tau)>0$,
and we are using the standard product notation
\begin{align*}
	\aqprod{z}{q}{n} 
		&= \prod_{j=0}^{n-1} (1-zq^j)
	,
	&\aqprod{z}{q}{\infty} 
		&= \prod_{j=0}^\infty (1-zq^j)
	,\\
	\aqprod{z_1,\dots,z_k}{q}{n} 
		&= \aqprod{z_1}{q}{n}\dots\aqprod{z_k}{q}{n}
	,
	&\aqprod{z_1,\dots,z_k}{q}{\infty} 
		&= \aqprod{z_1}{q}{\infty}\dots\aqprod{z_k}{q}{\infty}
	,\\
	\jacprod{z}{q}
		&= \aqprod{z,q/z}{q}{\infty}.	
	&
\end{align*}

Both the rank of partitions and overpartitions have been studied extensively.
In \cite{BringmannLovejoy} Bringmann and Lovejoy established for $z$
a root of unity that $\mathcal{O}(z;\tau)$ is the holomorphic part of a 
harmonic weak Maass form of weight $\frac{1}{2}$ and Dewar made certain refinements
in \cite{Dewar}. The work of Bringmann and Lovejoy in \cite{BringmannLovejoy}
was done in a fashion similar to the work of Bringmann and Ono in 
\cite{BringmannOno} for the rank of partitions, upon which Garvan \cite{Garvan3} has recently
made impressive improvements.
In \cite{LovejoyOsburn1} Lovejoy and Osburn gave formulas 
for the rank differences $\overline{N}(r,t,n)-\overline{N}(s,t,n)$,
for $t=3$ and $t=5$,
in terms of infinite products and generalized Lambert series. 
Determining these difference formulas is equivalent to
determining the $3$ dissection of $\mathcal{O}(\exp(2\pi i/3);\tau)$
and the $5$ dissection of $\mathcal{O}(\exp(2\pi i/5);\tau)$.
We will give 
similar formulas for $t=7$ and
for this we revisit $\mathcal{O}(z;q)$ as a harmonic weak Maass form.
We could also derive these results from identities between generalized Lambert
series as in \cite{LovejoyOsburn1}. The Lambert series method has the advantage
that one does not need to work out modular transformation formulas or introduce
harmonic Maass forms, but one must derive identities for all of the
rank differences. Using harmonic Maass forms has the advantage that one could 
prove identities for the rank differences individually.
Both methods require that we
must first guess the identities to prove them
and the proofs usually reduce to verifying an identity between modular forms.

With $\z{c}{}=\exp\left(\frac{2\pi i}{c}\right)$, we complete 
$\mathcal{O}(\z{c}{a};\tau)$ to a harmonic Maass form $\mathcal{M}(a,c)$ in 
Section 2. It 
turns out this harmonic Maass form is the sum of an easily understood modular form 
and a harmonic Maass form of Zwegers. With this we give explicit and compact 
transformation formulas under $\SLTwo$ in Sections 3 and 4. These transformation
formulas allow us to not only reprove that $\mathcal{O}(\z{c}{a};\tau)$ is the holomorphic
part of a harmonic Maass form of weight $\frac{1}{2}$ on $\Gamma_1(16c^2)$, but 
also determine a larger subgroup of $\SLTwo$ on which $\mathcal{M}(a,c)$ has a
rather simple multiplier. The contents of this are in Corollaries
\ref{CorOToM}, \ref{CorNTransformations1}, and \ref{PropPTransformations1}.
Additionally the transformation formulas for $\SLTwo$ allow us to give formulas and
bounds on the orders of our functions at the cusps
of $\SLTwo$. 
As an application of the transformation and order formulas, along with the
Valence formula for modular functions, we are able to prove the following 
theorem which gives the $7$-dissection of $\mathcal{O}(\z{7}{};\tau)$. It is 
by determining a large subgroup of $\SLTwo$ on which $\mathcal{M}(a,c)$
is not a harmonic Maass form, but has a simple multiplier system, that we 
can prove this identity by checking a small number of coefficients. In
particular the proof will be to check $110$ coefficients of an equivalent 
identity. As such we see it is possible to use Maass forms to find and prove
new concrete identities related to overpartitions.

\begin{theorem}\label{TheoremMain}
Let $\z{7}{}$ be a primitive seventh root of unity,
$J_a = \aqprod{q^a,q^{14-a}}{q^{14}}{\infty}$
for $1\le a\le 7$,
$J_0 = \aqprod{q^{14}}{q^{14}}{\infty}$, and
$A(x,y,z) = x\z{7}{}+y\z{7}{2}+z\z{7}{3}+z\z{7}{4}+y\z{7}{5}+x\z{7}{6}$.
Then
\begin{align*}
	\mathcal{O}(\z{7}{};\tau)
	&= 
		R_0(q^7)+qR_1(q^7)+q^2R_2(q^7)+q^3R_3(q^7)+q^4R_4(q^7)+q^5R_5(q^7)+q^6R_6(q^7)	
,
\end{align*}
where,
\begin{align*}
	R_0(q)
	&=
		\frac{A(-16,-8,-8)J_0}{J_1^3 J_2^3 J_3^3 J_4 J_5 J_6 J_7^2}
		+
		\frac{A(-1,1,-1)J_0}{J_1^3 J_2^3 J_3^3 J_5^3 J_6^2}	
		+
		\frac{A(1,-1,-1)J_0}{J_1^3 J_2^3 J_3^2 J_4^3 J_6^3}
		+
		\frac{A(-5,-1,-3)J_0}{J_1^3 J_2^2 J_3^3 J_4^3 J_5^3}
		+
		\frac{A(20,8,12)J_0}{J_1^3 J_2^3 J_3^2 J_4^2 J_5^3 J_7}
		\\&\quad
		+
		\frac{A(-3,-5,-1)qJ_0 J_4}{J_1^3 J_2^3 J_3^2 J_5^2 J_6^3 J_7^2}		
		+
		\frac{A(4,4,0)qJ_0}{J_1^3 J_2^3 J_4^3 J_5 J_6^3 J_7}		
		+
		\frac{A(3,5,1)q^2J_0 J_3^2}{J_1^3 J_2^3 J_4^3 J_5^2 J_6^3 J_7^2}
	,\\
	R_1(q)
	&=
		\frac{A(2,-4,2)J_0}{J_1^3 J_2^3 J_3^3 J_5^2 J_6^2 J_7}
		+
		\frac{A(18,6,10)J_0 J_5}{J_1^3 J_2^3 J_3^2 J_4^3 J_6^3 J_7}
		+
		\frac{A(-12,-2,-8)J_0}{J_1^3 J_2^3 J_3^2 J_4^2 J_5^2 J_7^2}	
		+
		\frac{A(18,6,10)qJ_0 J_4}{J_1^3 J_2^3 J_3^2 J_5 J_6^3 J_7^3}
		\\&\quad		
		+
		\frac{A(-18,-6,-10)q^2J_0 J_3^2}{J_1^3 J_2^3 J_4^3 J_5 J_6^3 J_7^3}
		+
		\frac{A(-10,-2,-6)J_0}{J_1^3 J_2^2 J_3^3 J_4^3 J_5^2 J_7}
		+
		\frac{A(-12,-2,-8)q J_0}{J_1^3 J_2^2 J_3 J_4^3 J_5^3 J_7^2}
	,\\
	R_2(q)
	&=
		\frac{A(12,10,6)J_0}{J_1^3 J_2^3 J_3^3 J_5 J_6^2 J_7^2}
		+
		\frac{A(-30,-6,-18)J_0}{J_1^3 J_2^3 J_3^2 J_4^2 J_5 J_7^3}
		+
		\frac{A(-14,-6,-10)J_0}{J_1^3 J_2^3 J_3^2 J_4 J_5^3 J_6 J_7}
		+
		\frac{A(16,2,10)J_0 J_6}{J_1^3 J_2^3 J_3 J_4^3 J_5^3 J_7^2}
		+
		\frac{A(1,-1,1)qJ_0}{J_1^3 J_2^3 J_3 J_5^2 J_6^2 J_7^3}		
		\\&\quad		
		+
		\frac{A(37,-1,25)J_0}{2 J_1^3 J_2^2 J_3^3 J_4^3 J_5 J_7^2}
		+
		\frac{A(-1,-3,-1)qJ_0}{J_1^3 J_2^2 J_3^2 J_4 J_5^2 J_6^2 J_7^2}
		+
		\frac{A(-2,2,-2)qJ_0}{J_1^3 J_2 J_3^3 J_4^3 J_6 J_7^3}
		+
		\frac{A(1,-1,1)J_0 J_3^3 J_4^2 J_5}{2 J_1^3 J_6^3 J_7^2}
		\\&\quad	
		+
		\frac{A(1,-1,1)J_0 J_2 J_3 J_4^3}{J_1^3 J_6 J_7^3}
		+
		\frac{A(-2,2,-2) J_0 J_2^2 J_4^3 J_5^2}{J_1^3 J_3 J_6^2 J_7^3}
	,\\
	R_3(q)
	&=
		\frac{A(8,-8,8)J_0}{J_1^3 J_2^3 J_3^3 J_6^2 J_7^3}
		+		
		\frac{A(-8,8,-8)J_0 J_5^3}{J_1^3 J_2^3 J_3^2 J_4^3 J_6^3 J_7^3}
		+
		\frac{A(-40,8,-32)J_0}{J_1^3 J_2^3 J_3^2 J_4 J_5^2 J_6 J_7^2}
		+
		\frac{A(28,-2,20)J_0}{J_1^3 J_2^3 J_3 J_4^3 J_5 J_6^3}	
		+
		\frac{A(20,6,12)qJ_0 J_3}{J_1^3 J_2^3 J_4^3 J_5^2 J_6^3 J_7}
		\\&\quad		
		+
		\frac{A(8,-8,8)J_0}{J_1^3 J_2^2 J_3^3 J_4^3 J_7^3}
		+
		\frac{A(-8,-6,-4)qJ_0}{J_1^3 J_2^2 J_3^2 J_5^3 J_6^3 J_7}
		+
		\frac{A(-8,0-6)q}{ J_0 J_7 }
		\sum_{n=-\infty}^\infty
		\frac{(-1)^n q^{7n(n+1)} }{1-q^{7n+2}}
	,\\
	R_4(q)
	&=	
		\frac{A(-4,-4,-2)J_0 J_4}{J_1^3 J_2^3 J_3^3 J_5 J_6^3 J_7^2}
		+
		\frac{A(2,-2,2)J_0}{J_1^3 J_2^3 J_3^2 J_5^3 J_6^2 J_7}
		+
		\frac{A(-6,2,-6)J_0}{J_1^3 J_2^3 J_3 J_4^3 J_6^3 J_7}
		+
		\frac{A(-10,-2,-8)qJ_0 J_4}{J_1^3 J_2^3 J_3 J_5^2 J_6^3 J_7^3}
		+
		\frac{A(4,4,2)q J_0 J_3}{J_1^3 J_2^3 J_4^3 J_5 J_6^3 J_7^2}
		\\&\quad		
		+
		\frac{A(10,2,8)q^2 J_0 J_3^3}{J_1^3 J_2^3 J_4^3 J_5^2 J_6^3 J_7^3}
		+
		\frac{A(10,2,6)J_0}{J_1^3 J_2^2 J_3^2 J_4^3 J_5^3 J_7}
	,\\		
	R_5(q)	
	&=		
		\frac{A(24,12,16)J_0 J_4}{J_1^3 J_2^3 J_3^3 J_6^3 J_7^3}
		+		
		\frac{A(8,8,6)J_0}{J_1^3 J_2^3 J_3^2 J_5^2 J_6^2 J_7^2}
		+
		\frac{A(-8,-8,-6)J_0}{J_1^3 J_2^3 J_3 J_4^2 J_5^2 J_7^3}
		+
		\frac{A(-24,-12,-16)qJ_0}{J_1^3 J_2^3 J_5^3 J_6^2 J_7^3}
		\\&\quad		
		+
		\frac{A(-24,-8,-18)J_0}{J_1^3 J_2^2 J_3^3 J_4 J_5^2 J_6^2 J_7}
		+
		\frac{A(0,-4,0)J_0}{J_1^3 J_2^2 J_3^2 J_4^3 J_5^2 J_7^2}
		+
		\frac{A(16,4,10)qJ_0}{J_1^3 J_2^2 J_4^3 J_5^3 J_7^3}
		+
		\frac{A(6,4,4)q}{ J_0 J_7 }
		\sum_{n=-\infty}^\infty
		\frac{(-1)^n q^{7n(n+1)} }{1-q^{7n+3}}
	,\\
	R_6(q)
	&=		
		\frac{A(-40,-8,-24)J_0}{J_1^3 J_2^3 J_3^2 J_5 J_6^2 J_7^3}		
		+
		\frac{A(2,4,0)J_0 J_4}{J_1^3 J_2^3 J_3^2 J_5^3 J_6^3 J_7}		
		+
		\frac{A(18,4,12)J_0 J_5^2}{J_1^3 J_2^3 J_3 J_4^3 J_6^3 J_7^3}		
		+
		\frac{A(40,8,24)J_0}{J_1^3 J_2^3 J_3 J_4 J_5^3 J_6 J_7^2}
		\\&\quad		
		+		
		\frac{A(-18,-4,-12)J_0}{J_1^3 J_2^3 J_4^3 J_5^2 J_6^3}
		+
		\frac{A(6,-4,4)J_0}{J_1^3 J_2^2 J_3^3 J_4 J_5 J_6^2 J_7^2}
		+
		\frac{A(-18,-4,-12)J_0}{J_1^3 J_2^2 J_3^2 J_4^3 J_5 J_7^3}
		+
		\frac{A(22,4,12)qJ_0}{J_1^3 J_2^2 J_3 J_4 J_5^2 J_6^2 J_7^3}
		\\&\quad		
		+
		\frac{A(8,2,4)}{ J_0 J_7 }
		\sum_{n=-\infty}^\infty
		\frac{(-1)^n q^{7n(n+1)} }{1-q^{7n+1}}
.
\end{align*}
\end{theorem}
One advantage to writing the identity in this form is that we can also read off
formulas for the rank differences,
\begin{align*}
	R_{r,s}(d;q)
	&=
	\sum_{n=0}^\infty \left(
		\overline{N}(r,7,7n+d) - \overline{N}(s,7,7n+d)
	\right)q^n
	.
\end{align*}
For this, we note 
\begin{align*}
	\mathcal{O}(\z{7}{};\tau)
	&=
		\sum_{n=0}^\infty
		\sum_{r=0}^6 N(r,7,n)\z{7}{r}
		q^n
	\\
	&=
		\sum_{n=0}^\infty
		\sum_{r=1}^6 (N(r,7,n)-N(0,7,n))\z{7}{r}
		q^n
	\\
	&=
		\sum_{r=1}^3
			(\z{7}{r}+\z{7}{7-r})	
		\sum_{d=0}^6
		\sum_{n=0}^\infty	
			(N(r,7,7n+d)-N(0,7,7n+d)) q^{7n+d}
	\\	
	&=
		\sum_{r=1}^3
		(\z{7}{r}+\z{7}{7-r})	
		\sum_{d=0}^6	
		q^d
		R_{r,0}(d;q^7)	
,
\end{align*}
where we have used that $1+\z{7}{}+\z{7}{2}+\dots+\z{7}{6}=0$
and $N(r,7,n)=N(7-r,7,n)$.
However $\z{7}{}$, $\z{7}{2}$, $\dots$, $\z{7}{6}$ are linearly independent 
over $\mathbb{Q}$, so if
\begin{align*}
	\mathcal{O}(\z{7}{};\tau)
	&=
	\sum_{r=1}^3
	(\z{7}{r}+\z{7}{7-r})	
	\sum_{d=0}^6
		q^d S(d;q^7)
,	
\end{align*}
and each $S(d;q)$ is a series in $q$ with rational coefficients, then
$S(d;q)=R_{r,0}(d;q)$.
As an example, from the formula for $R_0(q)$, we have that
\begin{align*}
	R_{1,0}(0;q)
	&=
		-
		\frac{16 J_0}{J_1^3 J_2^3 J_3^3 J_4 J_5 J_6 J_7^2}
		-
		\frac{J_0}{J_1^3 J_2^3 J_3^3 J_5^3 J_6^2}	
		+
		\frac{J_0}{J_1^3 J_2^3 J_3^2 J_4^3 J_6^3}
		-
		\frac{5 J_0}{J_1^3 J_2^2 J_3^3 J_4^3 J_5^3}
		+
		\frac{20 J_0}{J_1^3 J_2^3 J_3^2 J_4^2 J_5^3 J_7}
		\\&\quad
		-
		\frac{3 qJ_0 J_4}{J_1^3 J_2^3 J_3^2 J_5^2 J_6^3 J_7^2}		
		+
		\frac{4 qJ_0}{J_1^3 J_2^3 J_4^3 J_5 J_6^3 J_7}		
		+
		\frac{3 q^2J_0 J_3^2}{J_1^3 J_2^3 J_4^3 J_5^2 J_6^3 J_7^2}
.
\end{align*}

\section{Preliminaries}

We begin by defining the functions needed for the Maass forms.
For $u,v,z\in \mathbb{C}$, $\tau\in\mathcal{H}$, and
$u,v\not\in\mathbb{Z}+\tau\mathbb{Z}$ we have
\begin{align*}
	\vartheta(z;\tau)
	&=
		\sum_{n\in \mathbb{Z}+\frac{1}{2}}
		\exp\left( \pi in^2\tau + 2\pi in\left(z+\frac{1}{2}\right) \right)
	,\\
	\mu(u,v;\tau)
	&=
		\frac{\exp(\pi iu)}{\vartheta(v;\tau)}
		\sum_{n=-\infty}^\infty
		\frac{ (-1)^n \exp(\pi in(n+1)\tau + 2\pi inv) }
			{ 1 - \exp(2\pi in\tau + 2\pi iu ) }
.
\end{align*}
Next for $u,z\in\mathbb{C}$, $y=Im(\tau)$, and
$a=Im(u)/Im(\tau)$ we define
\begin{align*}
	E(z) 
	&=
		2\int_0^z \exp(-\pi w^2)dw
	,\\
	R(u;\tau)
	&=
		\sum_{n\in \mathbb{Z}+\frac{1}{2}}
		\left( \mbox{sgn}(n) - E( (n+a)\sqrt{2y} )  \right)
		(-1)^{n-\frac{1}{2}}	
		\exp( -\pi in^2\tau - 2\pi inu )	
.
\end{align*}
For $a,b\in\mathbb{R}$ we set
\begin{align*}
	g_{a,b}(\tau) 
	&=
	\sum_{n\in\mathbb{Z}+a}
	n\exp( \pi in^2\tau + 2\pi inb )
.
\end{align*}
Finally for $u,v\notin\mathbb{Z}+\tau\mathbb{Z}$ we set 
\begin{align*}
	\tilde{\mu}(u,v;\tau)
	&=
		\mu(u,v;\tau) + \frac{i}{2}R(u-v;\tau)	
	.
\end{align*}
In his revolutionary PhD thesis \cite{Zwegers},
Zwegers studied these functions and gave their transformation
formulas.

To work with $\mathcal{O}(z;\tau)$ as a Maass form, we relate
it to the functions studied by Zwegers,
rather than following the development of Bringmann and Lovejoy. 
This is similar to Mao's work on the $M_2$-rank for overpartitions \cite{Mao}. 
Relating functions to $\mu(u,v;\tau)$ is also how Hickerson and 
Mortenson studied mock theta functions in \cite{HickersonMortenson}.
While we should be able to derive our results from Bringmann and Lovejoy's work, we will see 
using the 
functions of Zwegers allows for less notation, keeps the transformation formulas
simple, and allows us to easily deduce the orders at cusps.
The initial step in relating $\mathcal{O}(z;\tau)$ to $\tilde{\mu}(u,v;\tau)$ is 
little more than rearranging fractions.

\begin{proposition}\label{PropRankToMu}
Let $z=\exp(\pi iu)$, then
\begin{align*}
	\mathcal{O}(z,\tau)
	&=
	\frac{(1-z)}{(1+z)}
	+
	2z\frac{(1-z)\aqprod{-q}{q}{\infty}\aqprod{q^2}{q^2}{\infty}^2}
		{(1+z)\aqprod{q}{q}{\infty}\jacprod{z^2}{q^2}}	
	-
	i2zq^{-1/4}\frac{(1-z)}{(1+z)}\mu(u,\tau;2\tau)	
	.
\end{align*}
\end{proposition}
\begin{proof}
This proposition is basically Lemma 2.1 of \cite{Mao2}, but in different 
notation. By \cite{Lovejoy2} we have
\begin{align}\label{EqPropRankToMu1}
	\mathcal{O}(z,\tau)
	&=
	\frac{\aqprod{-q}{q}{\infty}}{\aqprod{q}{q}{\infty}}
	\left(
		1
		+
		2\sum_{n=1}^\infty \frac{(1-z)(1-z^{-1})(-1)^nq^{n^2+n}}{(1-zq^n)(1-z^{-1}q^n)}
	\right)
	\nonumber\\
	&=
	\frac{\aqprod{-q}{q}{\infty}}{\aqprod{q}{q}{\infty}}
	\sum_{n=-\infty}^\infty \frac{(1-z)(1-z^{-1})(-1)^nq^{n^2+n}}{(1-zq^n)(1-z^{-1}q^n)}
	\nonumber\\
	&=
	\frac{\aqprod{-q}{q}{\infty}}{\aqprod{q}{q}{\infty}}
	\sum_{n=-\infty}^\infty \frac{z(1-z)(1-z^{-1})(-1)^nq^{n^2+n}}{(z-z^{-1})(1-zq^n)}
	-
	\frac{\aqprod{-q}{q}{\infty}}{\aqprod{q}{q}{\infty}}
	\sum_{n=-\infty}^\infty \frac{z^{-1}(1-z)(1-z^{-1})(-1)^nq^{n^2+n}}{(z-z^{-1})(1-z^{-1}q^n)}	
	\nonumber\\	
	&=
	\frac{\aqprod{-q}{q}{\infty}}{\aqprod{q}{q}{\infty}}
	\sum_{n=-\infty}^\infty \frac{z(1-z)(1-z^{-1})(-1)^nq^{n^2+n}}{(z-z^{-1})(1-zq^n)}
	+
	\frac{\aqprod{-q}{q}{\infty}}{\aqprod{q}{q}{\infty}}
	\sum_{n=-\infty}^\infty \frac{(1-z)(1-z^{-1})(-1)^nq^{n^2}}{(z-z^{-1})(1-zq^n)}	
	\nonumber\\
	&=
	\frac{\aqprod{-q}{q}{\infty}(1-z)}{\aqprod{q}{q}{\infty}(1+z)}
	\sum_{n=-\infty}^\infty \frac{(-1)^nq^{n^2}(1+2zq^n+z^2q^{2n})}{(1-z^2q^{2n})}
.
\end{align}
First by the $r=0$ and $s=1$ case of Theorem 2.1 of \cite{Chan} we have
\begin{align}\label{EqProductLambertSeries}
	\sum_{n=-\infty}^\infty \frac{(-1)^nq^{n^2+n}}{(1-z^2q^{2n})}
	&=
	\frac{\aqprod{q^2}{q^2}{\infty}^2}{\jacprod{z^2}{q^2}}
	.
\end{align}
Next by Proposition 1.3 of \cite{Zwegers} we have
\begin{align*}
	\vartheta(\tau;2\tau) 
	&= 
	-iq^{-\frac{1}{4}}\aqprod{q,q,q^2}{q}{\infty}
	=
	-iq^{-\frac{1}{4}}\frac{\aqprod{q}{q}{\infty}}{\aqprod{-q}{q}{\infty}}
	,\\
	\vartheta(-\tau;2\tau) 
	&= 
	iq^{-\frac{1}{4}}\frac{\aqprod{q}{q}{\infty}}{\aqprod{-q}{q}{\infty}}
.
\end{align*}
We find
\begin{align*}
	\mu(u,-\tau;2\tau)
	&=	
	\frac{z}{\vartheta(-\tau;2\tau)}
	\sum_{n=-\infty}^\infty \frac{(-1)^nq^{n^2}}{1-z^2q^{2n}}
	=
	-izq^{\frac{1}{4}}\frac{\aqprod{-q}{q}{\infty}}{\aqprod{q}{q}{\infty}}
	\sum_{n=-\infty}^\infty \frac{(-1)^nq^{n^2}}{1-z^2q^{2n}}
\end{align*}
and similarly
\begin{align*}
	\mu(u,-\tau;2\tau)
	&=	
	izq^{\frac{1}{4}}\frac{\aqprod{-q}{q}{\infty}}{\aqprod{q}{q}{\infty}}
	\sum_{n=-\infty}^\infty \frac{(-1)^nq^{n^2+2n}}{1-z^2q^{2n}}
.
\end{align*}
However, from Proposition 1.4 of \cite{Zwegers} we have that
\begin{align*}
	\mu(u,-\tau;2\tau)
	&=
	-z^2\mu(u,\tau;2\tau) - izq^{\frac{1}{4}}
.
\end{align*}
We see (\ref{EqPropRankToMu1}) then becomes
\begin{align*}
	\mathcal{O}(z;\tau)
	&=
	2z\frac{(1-z)\aqprod{-q}{q}{\infty}\aqprod{q^2}{q^2}{\infty}^2}
		{(1+z)\aqprod{q}{q}{\infty}\jacprod{z^2}{q^2}}	
	+
	\frac{(1-z)}{(1+z)}	
	\left(
		iz^{-1}q^{-\frac{1}{4}}\mu(u,-\tau;2\tau)
		-
		iz^{-1}q^{-\frac{1}{4}}\mu(u,\tau;2\tau)			
	\right)	
	\\
	&=
	\frac{(1-z)}{(1+z)}
	+
	2z\frac{(1-z)\aqprod{-q}{q}{\infty}\aqprod{q^2}{q^2}{\infty}^2}
		{(1+z)\aqprod{q}{q}{\infty}\jacprod{z^2}{q^2}}	
	-
	i2zq^{-\frac{1}{4}}\frac{(1-z)}{(1+z)}\mu(u,\tau;2\tau)	
.	
\end{align*}
\end{proof}

The term $\frac{(1-z)}{(1+z)}$ may seem out of place, however we will find it 
is necessary to accommodate $R(u-v;\tau)$.
We recall Theorem 1.16 of \cite{Zwegers} states
\begin{align*}
	R(a\tau-b)
	&=
	-\exp(\pi ia^2\tau - 2\pi ia(b+\tfrac{1}{2} )  )
	\int_{-\overline{\tau}}^{i\infty}
	\frac{g_{a+ \frac{1}{2} , b+ \frac{1}{2} }(z)}{\sqrt{-i(z+\tau)}}dz
,
\end{align*}
for $a\in(-\frac{1}{2},\frac{1}{2})$. 
Following the proof we find
that for $a=-\frac{1}{2}$ we instead have
\begin{align*}
	R(-\tfrac{\tau}{2}-b)
	&=
	\exp( \tfrac{\pi i\tau}{4}  +\pi ib  )
	-
	\exp( \tfrac{\pi i\tau}{4}  +\pi i(b+\tfrac{1}{2} )  )
	\int_{-\overline{\tau}}^{i\infty}
	\frac{g_{0, b+\frac{1}{2} }(z)}{\sqrt{-i(z+\tau)}}dz
.
\end{align*}

We then deduce the following.
\begin{corollary}\label{CorOToM}
Suppose $a$ and $c$ are integers, $c>0$, and $c\nmid 2a$. Then
\begin{align*}
	\mathcal{O}(\zeta_c^a;\tau)
	&=
	\frac{ 2\z{c}{a}(1-\z{c}{a})}{(1+\z{c}{a})}P(a,c;\tau)
	-
	\frac{i2\z{c}{a}(1-\z{c}{a})}{(1+\z{c}{a})}N(a,c;\tau)	
	+
	i\sqrt{2}\frac{(1-\z{c}{a})}{(1+\z{c}{a})}\int_{-\overline{\tau}}^{i\infty}
	\frac{g_{0, \frac{1}{2}-\frac{2a}{c} }(2z)}{\sqrt{-i(z+\tau)}}dz
,
\end{align*}
where
\begin{align*}
	N(a,c;\tau) 
	&=
		q^{-\frac{1}{4}}\tilde{\mu}(\tfrac{2a}{c},\tau;2\tau)
	,&
	P(a,c;\tau)
	&=
		\frac{\aqprod{-q}{q}{\infty}\aqprod{q^2}{q^2}{\infty} }
		{\aqprod{q}{q}{\infty} \jacprod{\z{c}{2a}}{q^2} }
.
\end{align*}
\end{corollary}

We set
\begin{align*}
	\mathcal{M}(a,c;\tau)
	&=
		\mathcal{O}(\zeta_c^a;\tau)
		-
		i\sqrt{2}\frac{(1-\z{c}{a})}{(1+\z{c}{a})}		
		\int_{-\overline{\tau}}^{i\infty}
		\frac{g_{0, \frac{1}{2}- \frac{2a}{c} }(2z)}{\sqrt{-i(z+\tau)}}dz
	\\	
	&=
	\frac{ 2\z{c}{a}(1-\z{c}{a}) }{(1+\z{c}{a}) }P(a,c;\tau)
	-
	\frac{i2\z{c}{a}(1-\z{c}{a})}{(1+\z{c}{a})}N(a,c;\tau)	
	.
\end{align*}
We do not handle the case when $z=-1$, however Bringmann and Lovejoy were able 
to handle this case in \cite{BringmannLovejoy}.

Next we consider the generalized Lambert series appearing in the
$R_3(q)$, $R_5(q)$, and $R_6(q)$ terms of the
dissection for $\mathcal{O}(\z{7}{};\tau)$.
Upon replacing $q$ by $q^7$ and multiplying by the appropriate power of $q$,
these are
\begin{align*}
	&\frac{q^{10}\aqprod{q^{98}}{q^{98}}{\infty}}
		{\aqprod{q^{49}}{q^{49}}{\infty}^2}
	\sum_{n=-\infty}^\infty
	\frac{ (-1)^n q^{49n(n+1)} }{ 1 - q^{49n+14} }
	,
	&\frac{q^{12}\aqprod{q^{98}}{q^{98}}{\infty}}
		{\aqprod{q^{49}}{q^{49}}{\infty}^2}
	\sum_{n=-\infty}^\infty
	\frac{ (-1)^n q^{49n(n+1)} }{ 1 - q^{49n+21} }
	,\\
	&
	\frac{q^{6}\aqprod{q^{98}}{q^{98}}{\infty}}
		{\aqprod{q^{49}}{q^{49}}{\infty}^2}
	\sum_{n=-\infty}^\infty
	\frac{ (-1)^n q^{49n(n+1)} }{ 1 - q^{49n+7} }
	.
\end{align*}
We note
\begin{align*}
	\sum_{n=-\infty}^\infty
	\frac{ (-1)^n q^{49n(n+1)} }{ 1 - q^{49n+7k} }
	&=
		\sum_{n=-\infty}^\infty
		\frac{ (-1)^n q^{49n(n+1)} (1+q^{49n+7k}) }
			{ 1 - q^{98n+14k} }
	\\
	&=
		\sum_{n=-\infty}^\infty
		\frac{ (-1)^n q^{49n(n+1)} }
			{ 1 - q^{98n+14k} }
		+
		q^{7k}
		\sum_{n=-\infty}^\infty
		\frac{ (-1)^n q^{49n(n+1)+49n}  }
			{ 1 - q^{98n+14k} }
	\\
	&=
		\frac{\aqprod{q^{98}}{q^{98}}{\infty}^2}{\jacprod{q^{14k}}{q^{98}}}
		+
		q^{7k}
		\sum_{n=-\infty}^\infty
		\frac{ (-1)^n q^{49n(n+1)+49n}  }
			{ 1 - q^{98n+14k} }
,
\end{align*}
but
\begin{align*}
	\mu(14k\tau,49\tau;98\tau)
	&=
	\frac{ iq^{7k+\frac{49}{4}} \aqprod{q^{98}}{q^{98}}{\infty}}
		{\aqprod{q^{49}}{q^{49}}{\infty}^2}
	\sum_{n=-\infty}^\infty
	\frac{ (-1)^n q^{49n(n+1)} }{ 1 - q^{49n+14k} }
,
\end{align*}
and so
\begin{align*}
	\frac{ q^{7k-k^2} \aqprod{q^{98}}{q^{98}}{\infty}  }
		{\aqprod{q^{49}}{q^{49}}{\infty}^2 }
	\sum_{n=-\infty}^\infty
	\frac{ (-1)^n q^{49n(n+1)} }
			{ 1 - q^{49n+7k} }
	&=
	\frac{q^{7k-k^2}\aqprod{q^{98}}{q^{98}}{\infty}^3 }
		{\aqprod{q^{49}}{q^{49}}{\infty}^2 \jacprod{q^{14k}}{q^{98}}}	
	-
	iq^{-k^2+7k-\frac{49}{4}} \mu(14k\tau, 49\tau; 98\tau)	
	.
\end{align*}
We then define the function $N_7(k;\tau)$, for $7\nmid k$, by
\begin{align*}
	N_7(k;\tau)
	&=
	q^{-k^2+7k-\frac{49}{4}} \tmu(14k\tau, 49\tau; 98\tau)	
	.
\end{align*}

Working with $\tilde{\mu}(u,v;\tau)$ is advantageous in that  
the transformation under the action of the modular group
$SL_2(\mathbb{Z})$ is known and quite elegant.
To work with the overpartition rank in terms of $\tmu$
we have introduced an extra product term, it is also possible to
not introduce this product and instead work with the function 
$f$ from chapter three of \cite{Zwegers}.
We reprove that $\mathcal{O}(z;q)$ is the holomorphic part of
a harmonic weak Maass form when
$z\not=\pm 1$ is a root of unity
and determine the order of the holomorphic part at cusps.  
By keeping track of the multipliers of our functions, we will be able
work on a fairly large subgroup of $\SLTwo$, which has relatively few cusps.

We recall $\SLTwo$ is the group of $2\times 2$ integer matrices with 
determinant $1$. The principal congruence subgroup of level $N$ is 
\begin{align*}
	\Gamma(N) &= 
	\left\{
		\TwoTwoMatrix{\alpha}{\beta}{\gamma}{\delta}\in \SLTwo
		:
		\alpha\equiv\delta\equiv 1\pmod{N}, 
		\gamma\equiv\delta\equiv 0\pmod{N}
	\right\}.
\end{align*}
A subgroup $\Gamma$ of $\SLTwo$ is called a congruence subgroup
if $\Gamma\supseteq\Gamma(N)$ for some $N$. Two congruence subgroups we will
use are
\begin{align*}
	\Gamma_0(N) &= 
	\left\{
		\TwoTwoMatrix{\alpha}{\beta}{\gamma}{\delta}\in \SLTwo
		: 
		\gamma\equiv 0\pmod{N}
	\right\}
	,\\
	\Gamma_1(N) &= 
	\left\{
		\TwoTwoMatrix{\alpha}{\beta}{\gamma}{\delta}\in \SLTwo
		: 
		\alpha\equiv\delta\equiv 1\pmod{N}, 		
		\gamma\equiv 0\pmod{N}
	\right\}
.
\end{align*}
We recall $\SLTwo$ acts on $\mathcal{H}$ via Mobius transformations, that is
$\TwoTwoMatrix{\alpha}{\beta}{\gamma}{\delta}\tau=\frac{\alpha\tau+\beta}{\gamma\tau+\delta}$.
Additionally we let $(A:\tau)=\gamma\tau+\delta$.

We recall a weakly holomorphic modular form of integral weight $k$
on a congruence subgroup $\Gamma$ of $\SLTwo$ is a holomorphic
function on $\mathcal{H}$ such that
\begin{enumerate}
\item if $A=\TwoTwoMatrix{\alpha}{\beta}{\gamma}{\delta}\in\Gamma$, then
$f(A\tau)=(\gamma\tau+\delta)^kf(\tau)$,
\item if $B\in\SLTwo$ then $(B:\tau)^{-k}F(Bz)$ has an expansion of the form
$\sum_{n=n_0}^\infty a_n\exp(2\pi inz/N)$.
\end{enumerate}
When $k$ is a half integer, we require $\Gamma\subset\Gamma_0(4)$ and replace 
$(1)$ with
$f(A\tau)=\Jac{\gamma}{\delta}^{2k}\epsilon(\delta)^{-2k}(\gamma\tau+\delta)^kf(\tau)$.
Here $\Jac{m}{n}$ is the Jacobi symbol extended to all integers $n$ by
\begin{align*}
	\Jac{0}{\pm 1} &= 1
	,\\
	\Jac{m}{n}
	&=
		\PieceTwo{\Jac{m}{|n|}}
		{-\Jac{m}{|n|}}
		{\mbox{ if } m>0,\mbox{ or, } m<0 \mbox{ and } n > 0   }
		{\mbox{ if } m < 0 \mbox{ and } n < 0}	
	,
\end{align*}
and $\epsilon(\delta)$ is $1$ when $\delta\equiv 1\pmod{4}$ and is $i$ otherwise.

A harmonic weak Maass form satisfies the transformation law in $(1)$, but 
the condition of holomorphic is replaced with being smooth and 
annihilated by the weight $k$ hyperbolic Laplacian operator,
\begin{align*}
	\Delta_k
	&= 
	-y^2\left( \frac{\partial^2}{\partial x^2}+\frac{\partial^2}{\partial y^2}  \right)
	+
	iky\left( \frac{\partial}{\partial x}+i\frac{\partial}{\partial y}  \right)
,
\end{align*}
where $\tau=x+iy$,
and condition $(2)$ is replaced with  $(B:\tau)^{-k}F(Bz)$ having at most linear
exponential growth as $z\rightarrow i\infty$.

If $f$ is a harmonic weak Maass form of weight $2-k$ on $\Gamma_1(N)$, then $f$
can be written as $f=f^{+} + f^-$, where $f^+$ and $f^-$ have expansions of the
form
\begin{align*}
	f^+(\tau)
	&=
	\sum_{n=n_0}^\infty a(n)q^n
	,&
	f^-(\tau)
	&=
	\sum_{n=1}^\infty b(n) \Gamma(k-1, 4\pi ny) q^{-n}
.
\end{align*}
Here $\Gamma$ is the incomplete Gamma function given by
$\Gamma(y,x) = \int_x^\infty e^{-t} t^{y-1}dt$.
We call $f^+$ the holomorphic part and $f^-$ the non-holomorphic part. The 
non-holomorphic part is often written instead as an integral of the form
\begin{align*}
	f^-(\tau)
	&=
	\int_{-\overline{\tau}}^\infty
	g(z) (-i(z+\tau))^{k-2} dz
\end{align*}
.

The rest of the article is organized as follows. In Section 3 we give the transformation
formulas for the action of $\SLTwo$ on $N(a,c;\tau)$ and prove $N(a,c;\tau)$ is 
a harmonic weak Maass form. In Section 4 we recognize $P(a,c;\tau)$ as a modular form
and give its transformation formulas. In Section 5 we give the transformation 
formulas for $N_7(k;\tau)$ and prove it is a harmonic weak Maass form. 
Additionally we recognize the products in the dissection of 
$\mathcal{O}(\z{7}{};\tau)$ as modular forms. In Section 6 we determine the 
orders at cusps for our various functions. In Section 7 we demonstrate that Theorem
\ref{TheoremMain} reduces to verifying a certain modular function is zero, which 
will follow by the valence formula.

\section{Transformations Formulas for $N(a,c;\tau)$ }

For a matrix  $A=\TwoTwoMatrix{\alpha}{\beta}{\gamma}{\delta}\in\SLTwo$,
we have, $\nu(A)$, the $\eta$-multiplier defined by
\begin{align*}
	\eta(A\tau) &= \nu(A)\sqrt{\gamma\tau+\delta} \,\eta(\tau)
,
\end{align*}
where $\eta(\tau)$ is Dedekind's eta-function,
\begin{align*}
	\eta(\tau) &= q^{\frac{1}{24}}\aqprod{q}{q}{\infty}
.
\end{align*}
A convenient form for the $\eta$-multiplier when $\gamma\not=0$, which can be found 
in \cite{Knopp}, is
\begin{align}
	\label{EqEtaMultipler}
	\nu(A)
	&=
	\left\{
	\begin{array}{ll}
		\big(\frac{\delta}{|\gamma|} \big)
		\exp\left(\frac{\pi i}{12}\left(
			(\alpha+\delta)\gamma - \beta\delta(\gamma^2-1) - 3\gamma		
		\right)\right)
		&
		\mbox{ if } \gamma \equiv 1 \pmod{2},
		\\				
		\Jac{\gamma}{\delta}
		\exp\left(\frac{\pi i}{12}\left(
			(\alpha+\delta)\gamma - \beta\delta(\gamma^2-1) + 3\delta - 3 - 3\gamma\delta		
		\right)\right)
		&
		\mbox{ if } \delta \equiv 1\pmod{2}.
	\end{array}
	\right.
\end{align}
For an integer $m$ we let
\begin{align*}
 	^mA &= \TwoTwoMatrix{\alpha}{m\beta}{\gamma/m}{\delta}
	,&
	_m A &= \TwoTwoMatrix{m\alpha}{\beta}{\gamma}{\delta/m} 
	.
\end{align*}

Our transformation formulas for $N(a,c;\tau)$ are easily deduced by the
transformations of $\tmu(u,v;\tau)$. The following essential properties are from
Theorem 1.11 of \cite{Zwegers}. If $k,l,m,n$ are integers then
\begin{align}
	\label{EqZTheorem1.11P1}
	\tmu\left(u+k\tau+l, v+m\tau+n; \tau \right)
	&=
	(-1)^{k+l+m+n}
	\exp\left(
		\pi i\tau(k-m)^2 + 2\pi i(k-m)(u-v)	
	\right)
	\tmu(u,v;\tau)
	.
\end{align}
If $A=\TwoTwoMatrix{\alpha}{\beta}{\gamma}{\delta}\in\SLTwo$ then
\begin{align}
	\label{EqZTheorem1.11P2}
	\tmu\left(
		\frac{u}{\gamma\tau+\delta},
		\frac{v}{\gamma\tau+\delta};
		A\tau		
	\right)
	&=
	\nu(A)^{-3}
	\exp\left(
		-\frac{\pi i\gamma(u-v)^2}{\gamma\tau+\delta}
	\right)	
	\sqrt{\gamma\tau+\delta},\
	\tmu\left(u, v;\tau\right)	
.
\end{align}

\begin{proposition}\label{PropNTransformations}
Suppose $a$ and $c$ are integers, $c>0$, and $c\nmid 2a$.
Let $A=\TwoTwoMatrix{\alpha}{\beta}{\gamma}{\delta}\in\SLTwo$.
If $\gamma$ is even then
\begin{align*}
	N(a,c;A\tau)
	&=
	\nu( ^2A)^{-3}
	(-1)^{\beta + \frac{\alpha-1}{2}}	
	\exp\left(-\pi i\left(
		\frac{2a^2\gamma\delta}{c^2} + \frac{2a}{c} + \frac{\alpha\beta}{2} - \frac{2a\delta}{c}		  
	\right)\right)
	\exp\left(-\pi i\tau\left(
		\frac{2a^2\gamma^2}{c^2} - \frac{2a\gamma}{c} + \frac{1}{2}		  
	\right)\right)
	\\&\quad
	\sqrt{\gamma\tau+\delta}
	\,\tilde{\mu}\left( \frac{2a\delta}{c}+\frac{2a\gamma\tau}{c} , \tau ; 2\tau \right)
.
\end{align*}
If $\gamma$ odd and $\delta$ even
\begin{align*}
	N(a,c;A\tau)
	&=
	2^{\frac{-1}{2}}
	\nu( _2A)^{-3}
	\exp\left(-\pi i\left(
		\frac{2a^2\gamma\delta}{c^2} - \frac{2a\beta\gamma}{c} + \frac{\alpha\beta}{2} 		  
	\right)\right)
	\exp\left(-\pi i\tau\left(
		\frac{2a^2\gamma^2}{c^2} - \frac{2a\alpha\gamma}{c} + \frac{\alpha^2}{2}		  
	\right)\right)
	\\&\quad
	\sqrt{\gamma\tau+\delta}
	\,\tilde{\mu}\left( 
		\frac{a\delta}{c} + \frac{a\gamma\tau}{c} , 
		\frac{\alpha\tau}{2} + \frac{\beta}{2} ; \frac{\tau}{2} \right)
.
\end{align*}
\end{proposition}
\begin{proof}
The proofs are lengthy, but straight-forward calculations from applying the 
transformation formula and identities for $\tilde{\mu}(u,v;\tau)$. 
First for $\gamma$ even we have $A\in\Gamma_0(2)$, so $^2A\in\SLTwo$,
and 
\begin{align*}
	2A\tau
	&=
	\frac{2\alpha\tau+2\beta}{\gamma\tau+\delta}
	=
	\frac{\alpha(2\tau)+2\beta}{\frac{\gamma}{2}(2\tau)+\delta}
	=
	 {^2A} (2\tau)
.
\end{align*}
We then apply (\ref{EqZTheorem1.11P2}) with $A\mapsto {^2A}$,
$\tau\mapsto 2\tau$,
$u = \frac{2a(\gamma\tau+\delta)}{c}$,
$v = (\alpha\tau+\beta)(\gamma\tau+\delta)$
to get
\begin{align*}
	N(a,c; A\tau)
	&=
		\exp\left(-\frac{\pi i A\tau}{2} \right)
		\tmu\left(
			\frac{2a}{c},
			A\tau;
			\frac{\alpha(2\tau)+2\beta}{\frac{\gamma}{2}(2\tau)+\delta}
		\right)
	\\
	&=
		\nu( ^2A)^{-3}
		\sqrt{\gamma\tau+\delta}
		\exp\left(-\frac{\pi i A\tau}{2} \right)
		\exp\left(   
			\frac{-\pi i \gamma}{2(\gamma\tau+\delta)}
			\left(   \frac{2a(\gamma\tau+\delta)}{c} - (\alpha\tau+\beta) \right)^2		
		\right)
		\\&\quad
		\tmu\left(
			\frac{2a(\gamma\tau+\delta)}{c},
			\alpha\tau+\beta;
			2\tau
		\right)
	\\
	&=
		\nu( ^2A)^{-3}
		\sqrt{\gamma\tau+\delta}
		\exp\left(-\frac{\pi i A\tau(1+\gamma(\alpha\tau+\beta))}{2} \right)
		\exp\left(
			-\pi i\tau\left(   \frac{2a^2\gamma^2}{c^2} - \frac{2a\alpha\gamma}{c}  \right)		
		\right)
		\\&\quad
		\exp\left(
			-\pi i\left(   \frac{2a^2\gamma\delta}{c^2} - \frac{2a\beta\gamma}{c}  \right)		
		\right)
		\tmu\left(
			\frac{2a(\gamma\tau+\delta)}{c},
			\alpha\tau+\beta;
			2\tau
		\right)
	\\
	&=
		\nu( ^2A)^{-3}
		\sqrt{\gamma\tau+\delta}
		\exp\left(
			-\pi i\tau\left(   \frac{2a^2\gamma^2}{c^2} - \frac{2a\alpha\gamma}{c} + \frac{\alpha^2}{2}  \right)		
		\right)
		\\&\quad
		\exp\left(
			-\pi i\left(   \frac{2a^2\gamma\delta}{c^2} - \frac{2a\beta\gamma}{c} + \frac{\alpha\beta}{2}  \right)		
		\right)
		\tmu\left(
			\frac{2a(\gamma\tau+\delta)}{c},
			\alpha\tau+\beta;
			2\tau
		\right)		
.
\end{align*}
We note that $\alpha$ is odd and so we apply (\ref{EqZTheorem1.11P1}) with $\tau\mapsto2\tau$,
$u = \frac{2a(\gamma\tau+\delta)}{c}$, $k=l=0$,
$v =\tau$, $m=\frac{\alpha-1}{2}$, $n=\beta$, and simplify to obtain
\begin{align*}
	N(a,c; A\tau)
	&=
	\nu( ^2A)^{-3}
	(-1)^{\beta + \frac{\alpha-1}{2}}	
	\exp\left(-\pi i\left(
		\frac{2a^2\gamma\delta}{c^2} + \frac{2a}{c} + \frac{\alpha\beta}{2} - \frac{2a\delta}{c}		  
	\right)\right)
	\exp\left(-\pi i\tau\left(
		\frac{2a^2\gamma^2}{c^2} - \frac{2a\gamma}{c} + \frac{1}{2}		  
	\right)\right)
	\\&\quad
	\sqrt{\gamma\tau+\delta}
	\,\tilde{\mu}\left( \frac{2a\delta}{c}+\frac{2a\gamma\tau}{c} , \tau ; 2\tau \right)
.
\end{align*}

When $\gamma$ is odd and $\delta$ is even, we instead have $_2A\in SL_2(\mathbb{Z})$ 
and $2A\tau= {_2A} (\frac{\tau}{2})$. We then apply (\ref{EqZTheorem1.11P2}) and 
simplify to obtain the result. We omit the details.
\end{proof}
We use the first case to determine on which subgroup of 
$\SLTwo$ that $N(a,c;\tau)$ is a Maass form and we use the second case to
determine orders at cusps.

\begin{corollary}
Suppose $a$ and $c$ are integers, $c>0$, and $c\nmid 2a$.
Let $A=\TwoTwoMatrix{\alpha}{\beta}{\gamma}{\delta}\in\Gamma_0(2)\cap\Gamma_0(c)$, then
\begin{align*}
	N(a,c;A\tau)
	&=
	\nu( ^2A)^{-3}
	(-1)^{\beta + \frac{\alpha-1}{2} + \frac{\alpha\gamma}{c}}	
	\exp\left(-\pi i\left(
		\frac{-2a^2\gamma\delta}{c^2} + \frac{2a(1-\delta)}{c} + \frac{\alpha\beta}{2}		  
	\right)\right)
	\sqrt{\gamma\tau+\delta}
	N(a\delta, c;\tau)
.
\end{align*}
\end{corollary}
\begin{proof}
This follows from Proposition \ref{PropNTransformations} and
(\ref{EqZTheorem1.11P1}) applied with $l=\frac{a\gamma}{c}$.
\end{proof}

\begin{corollary}\label{CorNTransformations1}
Suppose $a$ and $c$ are integers, $c>0$, and $c\nmid 2a$.
Let $A=\TwoTwoMatrix{\alpha}{\beta}{\gamma}{\delta}\in\Gamma_0(2)\cap\Gamma_0(c^2)\cap\Gamma_1(c)$,
then
\begin{align*}
	N(a,c;A\tau)
	&=
		\nu( ^2A)^{-3}\sqrt{\gamma\tau+\delta} (-1)^{\beta + \frac{\alpha-1}{2}} 
		i^{-\alpha\beta}	
		N(a,c;\tau)
\end{align*}
\end{corollary}
\begin{proof}
Using (\ref{EqZTheorem1.11P1}) we deduce that $N(a+c,c;\tau)=N(a,c;\tau)$, so 
with $\delta\equiv 1\pmod{c}$ we have $N(a\delta,c;\tau)=N(a,c;\tau)$.
\end{proof}

\begin{corollary}\label{CorNMaassForm}
Suppose $a$ and $c$ are integers, $c>0$, and $c\nmid 2a$.
Let $A=\TwoTwoMatrix{\alpha}{\beta}{\gamma}{\delta}\in
\Gamma_1(4)\cap\Gamma_1(c)\cap\Gamma_0(16)\cap\Gamma_0(c^2)$,
then
\begin{align*}
	N(a,c;A\tau)
	&=
	\Jac{\gamma}{\delta}
	\sqrt{\gamma\tau+\delta}
	N(a, c;\tau)
.
\end{align*}
In particular $N(a,c;\tau)$ is a harmonic weak Maass form of weight $\frac{1}{2}$ on
$\Gamma_1(4)\cap\Gamma_1(c)\cap\Gamma_0(16)\cap\Gamma_0(c^2)$.
\end{corollary}
\begin{proof}
Noting that $\alpha\equiv1 \pmod{4}$, for the transformation we must verify that
\begin{align}\label{EqNuJacIdent}
	\Jac{\gamma}{\delta}
	&=
	\nu( ^2A)^{-3} i^{\beta}	
.
\end{align}
Since $A\in\Gamma_0(16)$, we have $^2A\in\Gamma_0(8)$, and so
applying (\ref{EqEtaMultipler}) yields
\begin{align*}
	\nu( ^2A)^{-3} i^{\beta}	
	&=
	\Jac{\gamma/2}{\delta}
	\exp\left(\frac{-\pi i}{4}
	\left(	
		(\alpha+\delta)\tfrac{\gamma}{2}	
		-2\beta\delta( \tfrac{\gamma^2}{4} -1 )	
		+ 3\delta - 3 - \tfrac{3\gamma\delta}{2}
		-2\beta	
	\right)			
	\right)
	\\
	&=
	\Jac{\gamma}{\delta}
	\Jac{2}{\delta}
	\exp\left(\frac{-\pi i}{4}
	\left(	
		2\beta(\delta-1)
		+3(\delta-1)
	\right)			
	\right)
	\\
	&=
	\Jac{\gamma}{\delta}
	\Jac{2}{\delta}
	\exp\left( -3\pi i \frac{(\delta-1)}{4} \right)
	\\
	&=
	\Jac{\gamma}{\delta}
	\Jac{2}{\delta}
	 (-1)^{\frac{\delta-1}{4}}
	\\
	&=
	\Jac{\gamma}{\delta}
	(-1)^{\frac{\delta^2-1}{8} }
	(-1)^{\frac{\delta-1}{4}}
	\\	
	&=
	\Jac{\gamma}{\delta}
	(-1)^{\frac{ (\delta-1)(\delta+3)}{8} }
.
\end{align*}
But $\delta\equiv 1,5\pmod{8}$ so that 
$\frac{ (\delta-1)(\delta+3)}{8} $ is even. Thus
(\ref{EqNuJacIdent}) holds.

The growth condition at the cusps is satisfied by $N(a,c;\tau)$, as
if we take $\alpha/\gamma$ with $\mbox{gcd}(\alpha,\gamma)=1$,
then we can take a matrix $A=\TwoTwoMatrix{\alpha}{\beta}{\gamma}{\delta}\in\SLTwo$
and apply the transformations in Proposition $\ref{PropNTransformations}$.
Clearly we can apply the first case when $\gamma$ is even. We note when
$\gamma$ is odd that we may choose $\delta$ odd, as if it is not, then we
can replace $\beta$ by $\beta+\alpha$ and $\delta$ by $\delta+\gamma$.

Writing the weight $\frac{1}{2}$ hyperbolic Laplacian as
\begin{align*}
	\Delta_{\frac{1}{2}} 
	&= 
	-4y^{\frac{3}{2}}\frac{\partial}{\partial\tau}\sqrt{y}\frac{\partial}{\partial\overline{\tau}}
,
\end{align*}
we see $\Delta_{\frac{1}{2}}$ annihilates
$q^{-\frac{1}{4}}\mu( \tfrac{2a}{c} , \tau;2\tau)$ as it is holomorphic in $\tau$.
By Lemma 1.8 of \cite{Zwegers} we have
\begin{align*}
	\sqrt{y}\frac{\partial}{\partial\overline{\tau}}q^{-\frac{1}{4}}R( \tfrac{2a}{c} -\tau;2\tau)
	&=
	-i\exp(-\tfrac{\pi i\tau}{2} - \pi y)A(\overline{\tau})
	=
	-i\exp(-\tfrac{\pi i\overline{\tau}}{2} )A(\overline{\tau})
	,
\end{align*}
where $A(\overline{\tau})$ is a series giving a function holomorphic in $\overline{\tau}$.
Thus 
$\sqrt{y}\frac{\partial}{\partial\overline{\tau}}
q^{-\tfrac{1}{4}}R( \tfrac{2a}{c}-\tau;2\tau)$
is anti-holomorphic, and so 
$\Delta_{\frac{1}{2}}q^{-\frac{1}{4}}R( \tfrac{2a}{c}-\tau;2\tau)=0$. Therefore $\Delta_{\frac{1}{2}}$
annihilates $N(a,c;\tau)$.
\end{proof}

\begin{proposition}\label{PropNNonHolomorphicPart}
Suppose $a$ and $c$ are integers, $c>0$, and $c\nmid 2a$, then
the non-holomorphic part of $N(a,c;\tau)$ is given by
\begin{align*}
	\frac{i \z{c}{-a}}{2\sqrt{\pi}}
	\sum_{n=1}^\infty 
	(-1)^n \left(\z{c}{-2an}-\z{c}{2an} \right) 
	\Gamma(\tfrac{1}{2};4\pi yn^2) q^{-n^2}  
	.
\end{align*}
\end{proposition}
\begin{proof}
We recall 
\begin{align*}
	N(a,c;\tau)
	&= 
		q^{-\frac{1}{4}}\tmu(\tfrac{2a}{c}, t;2\tau)
	\\
	&=	
		q^{-\frac{1}{4}}\mu(\tfrac{2a}{c}, t;2\tau)
		+
		\frac{iq^{-\frac{1}{4}}}{2} R(\tfrac{2a}{c}-\tau;2\tau)
	\\
	&=
		q^{-\frac{1}{4}}\mu(\tfrac{2a}{c}, t;2\tau)
		+
		\frac{i\z{c}{-a}}{2}
		+
		\frac{\z{c}{-a}}{\sqrt{2}}\int_{-\overline{\tau}}^{i\infty}
		\frac{ g_{0,\frac{1}{2}-\frac{2a}{c}}(2z)  }{\sqrt{-i(z+\tau)}}
		dz
	.
\end{align*}
Thus the non-holomorphic part is
\begin{align*}
	\frac{\z{c}{-a}}{\sqrt{2}}
	\int_{-\overline{\tau}}^{i\infty}
		\frac{
			\sum_{n=-\infty}^\infty	
			n\exp(2\pi izn^2 + 2\pi in(\tfrac{1}{2}-\tfrac{2a}{c} )  )
		}{\sqrt{-i(z+\tau)}}
		dz
	&=
	\frac{\z{c}{-a}}{\sqrt{2}}
	\sum_{\substack{n=-\infty\\n\not=0}}^\infty
		(-1)^n n\z{c}{-2an}	
	\int_{-\overline{\tau}}^{i\infty}
		\frac{ e^{2\pi izn^2}}{\sqrt{-i(z+\tau)}}
	dz
.
\end{align*}
We exchange the order of the integral and series and use the substitution
$z=\frac{-t}{2\pi in^2}-\tau$ so that
\begin{align*}
	\frac{\z{c}{-a}}{\sqrt{2}}
	\sum_{\substack{n=-\infty\\n\not=0}}^\infty
		(-1)^n n\z{c}{-2an}	
	\int_{-\overline{\tau}}^{i\infty}
		\frac{ e^{2\pi izn^2}}{\sqrt{-i(z+\tau)}}
	dz
	&=
	-\frac{\z{c}{-a}}{\sqrt{2}}
	\sum_{\substack{n=-\infty\\n\not=0}}^\infty
		\frac{(-1)^n n\z{c}{-2an}}
		{2\pi in^2}	
	\int_{4\pi n^2y}^{\infty}
		e^{-t - 2\pi in^2\tau} 
		\sqrt{ \frac{  2\pi n^2}{t} }
	dt
	\\
	&=
	\frac{i\z{c}{-a}}{2\sqrt{\pi}}
	\sum_{\substack{n=-\infty\\n\not=0}}^\infty
		(-1)^n \mbox{sgn}(n) \z{c}{-2an} q^{-n^2}
	\Gamma(\tfrac{1}{2}; 4\pi n^2y)
	\\
	&=
	\frac{i\z{c}{-a}}{2\sqrt{\pi}}
	\sum_{n=1}^\infty
		(-1)^n (\z{c}{-2an} - \z{c}{2an}  )   q^{-n^2}
	\Gamma(\tfrac{1}{2}; 4\pi n^2y)
.
\end{align*}

\end{proof}


\section{Tranformations for $P(a,c;\tau)$}

The function $P(a,c;\tau)$ can be written in terms of known modular forms.
In particular.
As in \cite{KubertLang} we have the Klein forms given by
\begin{align*}
	t_{(a_1,a_2)}
	&=
	-q^{B_2(a_1)-\frac{1}{12} }\exp( \pi ia_2(a_1-1) )\frac{\jacprod{\zeta}{q}}{\aqprod{q}{q}{\infty}^2}	
	,
\end{align*}
where $B_2(x) = x^2-x+\frac{1}{6}$, $a_1$ and $a_2$ are rational, and 
$\zeta = \exp(2\pi i(a_1\tau + a_2))$. 
For  $A=\TwoTwoMatrix{\alpha}{\beta}{\gamma}{\delta}\in\SLTwo$,
\begin{align*}
	t_{(a_1,a_2)}(A\tau)
	&=
	(\gamma\tau+\delta)^{-1}t_{(a_1,a_2)\cdot A}(\tau)
,
\end{align*}
and for integers $b_1$ and $b_2$
\begin{align}\label{EqTShift}
	t_{(a_1+b_1,a_2+b_2)}(\tau)
	&=
	(-1)^{b_1b_2+b_1+b_2}\exp(-\pi i(b_1a_2-b_2a_1)) t_{(a_1,a_2)}(\tau)	
.
\end{align}
Additionally, $t_{(a_1,a_2)}(\tau)$ is holomorphic on $\mathcal{H}$ and has
no zeros nor poles on $\mathcal{H}$. Thus $t_{(a_1,a_2)}(\tau)$ is a
modular form of weight $-1$ on some subgroup of $\SLTwo$.

We note that
\begin{align*}
	P(a,c;\tau)
	&=
	\frac{-\z{c}{a}\eta(2\tau)}{\eta(\tau)^2 \,\, t_{0,\frac{2a}{c}}(2\tau) }
	.
\end{align*}

\begin{proposition}\label{PropPTransformations1}
Suppose $a$ and $c$ are integers, $c>0$, and $c\nmid 2a$.
Let $A=\TwoTwoMatrix{\alpha}{\beta}{\gamma}{\delta}\in
\Gamma_0(2)\cap\Gamma_0(c^2)\cap\Gamma_1(c)$, then
\begin{align*}
	P(a,c;A\tau)
	&=
		\frac{\nu( ^2A)}{\nu(A)^2}
		\sqrt{\gamma\tau+\delta}
		\,
		P(a,c;\tau)
	=
		\nu( ^2A)^{-2}
		(-1)^{\beta+\frac{\alpha-1}{2}}
		i^{-\alpha\beta}		
		\sqrt{\gamma\tau+\delta}
		\,
		P(a,c;\tau)				
.
\end{align*}
\end{proposition}
\begin{proof}
We have
\begin{align*}
	P(a,c;A\tau)
	&=
		\frac{-\z{c}{a}\nu( ^2A)}{\nu(A)^2}
		\sqrt{\gamma\tau+\delta}
		\,
		\frac{\eta(2\tau)}{\eta(\tau)^2 \,\, t_{\frac{a\gamma}{c},\frac{2a\delta}{c}}(2\tau) }
	\\
	&=
		\frac{-\z{c}{a}\nu( ^2A)}{\nu(A)^2}
		\sqrt{\gamma\tau+\delta}
		\,
		\frac{\eta(2\tau)}{\eta(\tau)^2 \,\, t_{0,\frac{2a}{c}}(2\tau) }
	\\
	&=
		\frac{\nu( ^2A)}{\nu(A)^2}
		\sqrt{\gamma\tau+\delta}
		\,
		P(a,c;\tau)
,
\end{align*}
where in the second line we have applied (\ref{EqTShift}) with 
$a_1=0$, $a_2=\frac{2a}{c}$, $b_1=\frac{a\gamma}{c}$, and 
$b_2=\frac{2a(\delta-1)}{c}$.

The second identity of the Proposition follows by a lengthy calculation with
(\ref{EqEtaMultipler}) to verify that 
on $\Gamma_0(2)$ we indeed have
$\frac{\nu( {^2}A )}{\nu(A)^2}=\nu({^2}A)^{-3}(-1)^{\beta+\frac{\alpha-1}{2}}i^{-\alpha\beta}$.
We omit these calculations, as they are rather long but straightforward.

\end{proof}

\begin{proposition}
Suppose $a$ and $c$ are integers, $c>0$, and $c\nmid 2a$.
Let $A=\TwoTwoMatrix{\alpha}{\beta}{\gamma}{\delta}\in
\Gamma_1(4)\cap\Gamma_1(c)\cap\Gamma_0(16)\cap\Gamma_0(c^2)$, then
\begin{align*}
	P(a,c;A\tau)
	&=
	\Jac{\gamma}{\delta}
	\sqrt{\gamma\tau+\delta}
	P(a, c;\tau)
.
\end{align*}
In particular $P(a,c;\tau)$ is a weight $\frac{1}{2}$ modular form on 
$\Gamma_1(4)\cap\Gamma_1(c)\cap\Gamma_0(16)\cap\Gamma_0(c^2)$.
\end{proposition}
\begin{proof}
As in the proof of Corollary \ref{CorNMaassForm} we have
$\nu({^2A})^{-3} i^{\beta}	=\Jac{\gamma}{\delta}$, so the identity holds.
\end{proof}

\begin{corollary}\label{CorScriptMMaassForm}
Suppose $a$ and $c$ are integers, $c>0$, and $c\nmid 2a$,
then 
$\mathcal{M}(a,c;\tau)$ is a harmonic weak Maass form of weight 
$\frac{1}{2}$
on $\Gamma_1(4)\cap\Gamma_1(c)\cap\Gamma_0(16)\cap\Gamma_0(c^2)$.
\end{corollary} 

\begin{proposition}\label{PropScriptMNonHolomorphicPart}
Suppose $a$ and $c$ are integers, $c>0$, and $c\nmid 2a$.
Then the non-holomorphic part of $\mathcal{M}(a,c;\tau)$ is given by
\begin{align*}
	\frac{1}{\sqrt{\pi}}
	\frac{(1-\z{c}{a})}{(1+\z{c}{a})}
	\sum_{n=1}^\infty 
	(-1)^n \left(\z{c}{-2an}-\z{c}{2an} \right) 
	\Gamma(\tfrac{1}{2};4\pi yn^2) q^{-n^2}  
\end{align*}
\end{proposition}


\section{Transformations for $N_7(k;\tau)$ and Additional Products}

We work with a more general function than $N_7(k;\tau)$.
For integers $a$ and $c$, $c>0$, and $c\nmid a$, we define 
\begin{align*}
	M(a,c;\tau)
	&=
	q^{-\frac{1}{2}\left(\frac{a}{c}-\frac{1}{2}\right)^2}
	\tmu\left(\frac{a\tau}{c}, \frac{\tau}{2}; \tau \right)
.
\end{align*}
We see that $N_7(k;\tau) = M(k,7;98\tau)$.

\begin{proposition}\label{PropMTransformation}
Suppose $a$ and $c$ are integers, $c>0$, and $c\nmid a$.
Let $A=\TwoTwoMatrix{\alpha}{\beta}{\gamma}{\delta}\in\SLTwo$, then
\begin{align*}
	M(a,c;A\tau)
	&=
	\nu(A)^{-3}
	\exp\left(-\pi i\alpha\beta\left(\tfrac{a}{c}-\tfrac{1}{2}\right)^2  \right)
	\exp\left(-\pi i\alpha^2\tau \left(\tfrac{a}{c}-\tfrac{1}{2}\right)^2 \right)
	\sqrt{\gamma\tau+\delta}
	\,\tmu\left( \frac{a\alpha\tau+a\beta}{c}, \frac{\alpha\tau+\beta}{2}; \tau \right)
.
\end{align*}
\end{proposition}
\begin{proof}
As with the transformations for $N(a,c;\tau)$, this follows from
(\ref{EqZTheorem1.11P2}) and elementary rearrangements.
\end{proof}

\begin{proposition}\label{PropN7kTransformation1}
Suppose $k$ is an integer and $7\nmid k$.
Let $A=\TwoTwoMatrix{\alpha}{\beta}{\gamma}{\delta}\in\Gamma_1(14)\cap\Gamma_0(98)$, 
then
\begin{align*}
	N_7(k;\tau)
	&=
		\nu( ^2A)^{-3}
		(-1)^{\beta+\frac{\alpha-1}{2}}
		i^{-\alpha\beta}
		\sqrt{\gamma\tau+\delta}		
		N_7(k;\tau)
.
\end{align*}
\end{proposition}
\begin{proof}
By Proposition \ref{PropMTransformation} we have
\begin{align*}
	N_7(k;A\tau)
	&=
	M(k,7; {^{98}A} \,98\tau)	
	\\
	&=
	\nu( ^{98}A)^{-3}
	i^{-\alpha\beta}
	\exp\left(-\pi i98\alpha^2\tau \left(\tfrac{k}{7}-\tfrac{1}{2}\right)^2 \right)
	\sqrt{\gamma\tau+\delta}
	\,\tmu\left( 14k\alpha\tau + 14k\beta  , 49\alpha\tau+49\beta   ; 98\tau \right)
	\\
	&=
	\nu( ^{98}A)^{-3}
	(-1)^{\beta}
	\sqrt{\gamma\tau+\delta}
	\exp\left(-\pi i98\alpha^2\tau \left(\tfrac{k}{7}-\tfrac{1}{2}\right)^2 \right)
	\,\tmu\left( 14k\alpha\tau   , 49\alpha\tau   ; 98\tau \right)
.
\end{align*}
But $\alpha\equiv 1\pmod{14}$, so 
applying (\ref{EqZTheorem1.11P1})
with
$\tau\mapsto 98\tau$,
$u=14k\tau$,
$v=49\tau$,
$k= \frac{k(\alpha-1)}{7}$
$l=0$,
$m=\frac{\alpha-1}{2}$,
$n=0$,
and simplifying yields
\begin{align*}
	\tmu\left( 14k\alpha\tau   , 49\alpha\tau   ; 98\tau \right)
	&=
	\tmu\left( 14k\tau + \frac{k(\alpha-1)}{7}98\tau   , 49\tau+\frac{(\alpha-1)}{2}98\tau  ; 98\tau \right)
	\\
	&=
	(-1)^{ \frac{\alpha-1}{2}}	
	\exp\left( 98\pi i\tau \alpha^2 \left(\tfrac{k}{7}-\tfrac{1}{2} \right)^2 
		-98\tau\left(\tfrac{k}{7}-\tfrac{1}{2}\right)^2
	\right)	
	\tmu\left( 14k\tau , 49\tau; 98\tau \right)
.
\end{align*}
Thus
\begin{align*}
	N_7(k;\tau)
	&=
		\nu( {^{98}A})^{-3}
		(-1)^{\beta+\frac{\alpha-1}{2}}
		i^{-\alpha\beta}
		\sqrt{\gamma\tau+\delta}		
		N_7(k;\tau)
.
\end{align*}
However, by Theorem 1.64 of \cite{Ono1} we find that 
$\frac{\eta(98\tau)^3}{\eta(2\tau)^3}$ is a modular function on
$\Gamma_0(98)$, so that
$\nu( ^{98}A)^3 = \nu( ^2A)^3$ for $A\in\Gamma_0(98)$.
\end{proof}

\begin{corollary}\label{CorN7MaassForm}
Suppose $k$ is in integer and $7\nmid k$.
Let $A=\TwoTwoMatrix{\alpha}{\beta}{\gamma}{\delta}\in
\Gamma_1(28)\cap\Gamma_0(784)$, then
\begin{align*}
	N_7(k;A\tau)
	&=
	\Jac{\gamma}{\delta}\sqrt{\gamma\tau+\delta}N_7(k;\tau)
.
\end{align*}
In particular, $N_7(k;\tau)$ is a harmonic weak Maass form
of weight $\frac{1}{2}$ on
$\Gamma_1(28)\cap\Gamma_0(784)$.
\end{corollary}
\begin{proof}
As in the proof of Corollary, \ref{CorNMaassForm}, on 
$\Gamma_1(28)\cap\Gamma_0(784)$ we have
$\nu( ^2A)^{-3} (-1)^{\beta+\frac{\alpha-1}{2}} i^{-\alpha\beta} = \Jac{\gamma}{\delta}$,
so that the transformation is correct. 

One can again use Lemma 1.8 of \cite{Zwegers} to see that the weight 
$\frac{1}{2}$ hyperbolic Laplacian annihilates $N_7(k;\tau)$.
For the condition at the cusps, one examines the transformations of
$M(a,c;98\tau)$ under $\SLTwo$, which we do in 
Proposition \ref{PropMOrders} of the next section to compute orders at cusps.

\end{proof}

\begin{proposition}\label{PropN7kNonHolomorphicPart}
For $1\le k\le6$, the non-holomorphic part of $N_7(k;\tau)$ is given by
\begin{align*}
	\frac{i  }{2\sqrt{\pi} }	
	\sum_{n=0}^\infty
	(-1)^n 
	q^{-(7n+k)^2}		
	\Gamma(\tfrac{1}{2}; 4\pi y(7n+k)^2)
	-
	\frac{i  }{2\sqrt{\pi} }	
	\sum_{n=1}^\infty
	(-1)^n 
	q^{-(7n-k)^2}		
	\Gamma(\tfrac{1}{2}; 4\pi y(7n-k)^2)
\end{align*}
\end{proposition}
\begin{proof}
We note that
\begin{align*}
	N_7(k;\tau)
	&= 
		q^{-(k-7/2)^2}\tmu(14k\tau,49\tau;98\tau)
	\\
	&=
		q^{-(k-7/2)^2}\mu(14k\tau,49\tau;98\tau)
		+
		\frac{iq^{-(k-7/2)^2}}{2}	R\left( \left(\tfrac{k}{7}-\tfrac{1}{2}\right)98\tau ;98\tau \right)
	\\
	&=	
		q^{-(k-7/2)^2}\mu(14k\tau,49\tau;98\tau)
		+
		\frac{\z{14}{-k}}{2}	
		\int_{-\overline{98\tau}}^{i\infty} 
			\frac{ g_{\frac{k}{7},\frac{1}{2} }(z)}{\sqrt{-i(z+98\tau)}} dz
	\\
	&=
		q^{-(k-7/2)^2}\mu(14k\tau,49\tau;98\tau)
		+
		\frac{7 \z{14}{-k} }{\sqrt{2} }	
		\int_{-\overline{\tau}}^{i\infty} 
			\frac{ g_{\frac{k}{7},\frac{1}{2} }(98z)}{\sqrt{-i(z+\tau)}} dz
	,
\end{align*}			
where the second to last line follows from Theorem 1.16 of \cite{Zwegers}.			
Thus the non-holomorphic part of $N_7(k;\tau)$ is given by
\begin{align*}
	\frac{7 \z{14}{-k} }{\sqrt{2} }
	\int_{-\overline{\tau}}^{i\infty}	
	\frac{	
		\sum_{n=-\infty}^\infty
		\left(n + \tfrac{k}{7} \right)
		\exp\left( 98\pi iz(n+ \tfrac{k}{7} )^2 + \pi i\left( n+ \tfrac{k}{7}  \right) \right)
	}{\sqrt{-i(z+\tau)}} dz
.
\end{align*}
We exchange the integral with the series, use the substitution
$z= \frac{-t}{98\pi i(n+ \frac{k}{7} )^2} -\tau$, and simplify
to find that 
\begin{align*}
	&\frac{7 \z{14}{-k} }{\sqrt{2} }
	\int_{-\overline{\tau}}^{i\infty}	
	\frac{	
		\sum_{n=-\infty}^\infty
		\left(n + \tfrac{k}{7} \right)
		\exp\left( 98\pi iz(n+\tfrac{k}{7} )^2 + \pi i\left( n+ \tfrac{k}{7}  \right) \right)
	}{\sqrt{-i(z+\tau)}} dz
	\\
	&=
		\frac{i  }{2\sqrt{\pi} }	
		\sum_{n=-\infty}^\infty
		(-1)^n
		\mbox{sgn}(7n + k )		 
		q^{-(7n+k)^2}		
		\Gamma(\tfrac{1}{2}; 4\pi y(7n+k)^2)
	\\
	&=
		\frac{i  }{2\sqrt{\pi} }	
		\sum_{n=0}^\infty
		(-1)^n 
		q^{-(7n+k)^2}		
		\Gamma(\tfrac{1}{2}; 4\pi y(7n+k)^2)
		-
		\frac{i  }{2\sqrt{\pi} }	
		\sum_{n=1}^\infty
		(-1)^n 
		q^{-(7n-k)^2}		
		\Gamma(\tfrac{1}{2}; 4\pi y(7n-k)^2)
.
\end{align*}
\end{proof}


We will write the additional products in terms of the function
$f_{N,\rho}$, defined for integers $N$ and $\rho$ with
$N\ge 1$ and $N\nmid\rho$ by
\begin{align*}
	f_{N,\rho}(\tau)
	&=
	q^{\frac{ (N-2\rho)^2 }{8N}}
	\aqprod{q^{\rho},q^{N-\rho},q^N}{q^N}{\infty}
.
\end{align*}
The transformations of $f_{N,\rho}(\tau)$ were studied by Biagioli in \cite{Biagioli}.
Lemma 2.1 of \cite{Biagioli} states that for
$A=\TwoTwoMatrix{\alpha}{\beta}{\gamma}{\delta}\in\Gamma_0(N)$ we have
\begin{align*}
	f_{N,\rho}(A\tau)
	&=
	\nu( ^NA)^3
	(-1)^{ \rho\beta + \Floor{\rho\alpha/N}+\Floor{\rho/N} }
	\exp\left( \frac{\pi i \alpha\beta\rho^2}{N}  \right)		
	\sqrt{\gamma\tau+\delta}
	f_{N,\rho\alpha}(\tau)		
	,
\end{align*}
where $\Floor{x}$ is the greatest integer less than or equal to $x$.
We note that
$f_{N,\rho}(\tau)=f_{N,N+\rho}(\tau)=f_{N,-\rho}(\tau)$.

\begin{proposition}\label{PropExtraProductsTransformation1}
If $\displaystyle F(\tau)=\eta(98\tau)^{r_0}\prod_{k=1}^7f_{98,7k}(\tau)^{r_k}$,
and $A\in\Gamma_1(14)\cap\Gamma_0(98)$ then
\begin{align*}
	F(A\tau)
	&=
	\nu( ^{98}A)^{r_0+3R}
	(-1)^{ \left(\frac{\alpha-1}{2}+\beta\right) S }
	\, i^{\alpha\beta S}
	(\gamma\tau+\delta)^{\frac{r_0+R}{2}}
	F(\tau)
,
\end{align*}
where $R=r_1+r_2+r_3+r_4+r_5+r_6+r_7$ and $S = r_1+r_3+r_5+r_7$.
\end{proposition}
\begin{proof}
This follows from the transformation formula for $f_{N,\rho}$ after a few
simplifications.
We note that $7k\alpha\equiv 7k\pmod{98}$,
and by writing
$\alpha = 1 + 14\frac{(\alpha-1)}{14}$ we see that
\begin{align*}
	\frac{7k\alpha}{98}
	&=
	\frac{7k}{98} + \frac{k(\alpha-1)}{14}	
	,\\
	\left\lfloor  \frac{7k\alpha}{98} \right\rfloor
	&=
	\frac{k(\alpha-1)}{14}	
	\equiv
	\frac{k(\alpha-1)}{2}	\pmod{2}
.
\end{align*}
Thus
\begin{align*}
	F(A\tau)
	&=
		\nu( {^{98}A} ) ^{r_0+3(r_1+r_2+r_3+r_4+r_5+r_6+r_7)}
		(-1)^{ (\beta+\frac{\alpha-1}{2}) \sum_{k=1}^7 kr_k  }		
		\exp\left(\frac{ \pi i\alpha\beta  }{2}\sum_{k=1}^7 k^2 r_k \right)		
	\\
	&=
		\nu( {^{98}A} ) ^{r_0+3(r_1+r_2+r_3+r_4+r_5+r_6+r_7)}
		(-1)^{ (\beta+\frac{\alpha-1}{2}) (r_1+r_3+r_5+r_7)  }		
		i^{ \alpha\beta(r_1+r_3+r_5+r_7) }		
	.
\end{align*}

\end{proof}

\begin{corollary}\label{CorExtarProductsModularForm}
If $F(\tau)$ is as in Proposition \ref{PropExtraProductsTransformation1},
$r_0 + R = 1$, $r_0+3R\equiv -3\pmod{24}$,
$1+S\equiv 0\pmod{4}$,
and $A=\TwoTwoMatrix{\alpha}{\beta}{\gamma}{\delta}
\in\Gamma_1(28)\cap\Gamma_0(784)$ then
\begin{align*}
	F(A\tau)
	&=
	\Jac{\gamma}{\delta}
	\sqrt{\gamma\tau+\delta}
	F(\tau)	
.
\end{align*}
In particular $F(\tau)$ is a modular form of weight $\frac{1}{2}$
on $\Gamma_1(28)\cap\Gamma_0(784)$.
\end{corollary}
\begin{proof}
We note
\begin{align*}
	F(A\tau) 
	&= 
	\nu( ^{98}A )^{-3}
	(-1)^{\beta}
	i^{-\beta}
	\sqrt{\gamma\tau+\delta}F(\tau)
,
\end{align*}
but as in the proof Corollary \ref{CorN7MaassForm} we know
$\nu( ^{98}A)^{-3}=\nu( {^2}A)^{-3}$ and
as in the proof of Corollary \ref{CorNMaassForm}
we know
$\Jac{\gamma}{\delta}=\nu( ^2A)^{-3} i^{\beta}$.
\end{proof}

\section{Orders at Cusps}

For a non-negative real number $x$, we let $\Floor{x}$ denote the greatest 
integer less than or equal to $x$ and $\Fractional{x}$ the fractional part of 
$x$. That is, 
$x=\Floor{x}+\Fractional{x}$, $\Floor{x}\in\mathbb{Z}$, and 
$0\le\Fractional{x}<1$.

\begin{proposition}\label{PropMuOrder}
If $u,v,w,x\in\mathbb{R}$ and $0\le u,w<1$, then the lowest power of $q$ appearing in the expansion
of $\mu(u\tau+v,w\tau+x;\tau)$ is $\nu(u,w)$, where
\begin{align*}
	\nu(u,w)
	&=
	\left\{
	\begin{array}{cc}
		\frac{u+w}{2}-\frac{1}{8} & \mbox{ if } u+w \le 1
		,\\\\
		\frac{7}{8}-\frac{u+w}{2} & \mbox{ if } u+w > 1.
	\end{array}\right.
\end{align*}
\end{proposition}
\begin{proof}
We have
\begin{align*}
	\mu(u\tau+v,w\tau+x;\tau)
	&=
	\frac{\exp(\pi iv)q^{\frac{u}{2}}}
		{\vartheta(w\tau+x;\tau  )}
	\sum_{n=-\infty}^\infty
	\frac{ (-1)^n q^{\frac{n(n+1)}{2} + nw  } \exp(2\pi inx) }  
	{1 - \exp(2\pi iv)q^{u+n}    }
.
\end{align*}
But
\begin{align*}
	\frac{1}{\vartheta(w\tau+x;\tau  )}
	&=
	\frac{  i\exp(\pi ix)q^{\frac{w}{2}-\frac{1}{8}}	}
	{ \aqprod{q,\exp(2\pi ix)q^w, \exp(-2\pi ix)q^{1-w}}{q}{\infty} }	
	\\
	&=
	\left\{
	\begin{array}{ll}
		\frac{  i\exp(\pi ix)q^{-\frac{1}{8}}  }{(1-\exp(2\pi ix))}(1+\dots)	
			&	
			\mbox{ if } w=0,
		\\
		 i\exp(\pi ix)q^{\frac{w}{2}-\frac{1}{8}}  (1+\dots)	
			&	
			\mbox{ if } w\not=0,
	\end{array}		
	\right.
\end{align*}
and
\begin{align*}
	&\sum_{n=-\infty}^\infty
	\frac{ (-1)^n q^{\frac{n(n+1)}{2} + nw  } \exp(2\pi inx) }  
		{1 - \exp(2\pi iv)q^{u+n} }
	\\
	&=
		\frac{1}{1-\exp(2\pi iv)q^{u}}	
		+
		\sum_{n=1}^\infty
		\frac{ (-1)^n q^{\frac{n(n+1)}{2} + nw  } \exp(2\pi inx) }  
			{1 - \exp(2\pi iv)q^{u+n} }
		\\&\quad		
		-
		\exp(-2\pi iv)q^{-u}
		\sum_{n=1}^\infty
		\frac{ (-1)^n q^{\frac{n(n+1)}{2} - nw  } \exp(-2\pi inx) }  
			{1 - \exp(-2\pi iv)q^{n-u} }
	\\
	&=
		\frac{1}{1-\exp(2\pi iv)q^{u}}	
		-
		\exp(2\pi ix)q^{1+w}(1+\dots)
		+
		\exp(-2\pi i(v+x))q^{1-u-w}(1+\dots)
	\\
	&=
	\left\{\begin{array}{ll}
		\frac{1}{1-\exp(2\pi iv)}(1+\dots)
			&
			\mbox{ if } u = 0,
		\\
		1+\dots
			&
			\mbox{ if } u \not=0 \mbox{ and } u+w < 1,
		\\
		(1+\exp(-2\pi i(v+x)))(1+\dots)
			&
			\mbox{ if } u+w = 1,
		\\
		\exp(-2\pi i(v+x)) q^{1-u-w} (1+\dots)
			&
			\mbox{ if } u+w > 1.
	\end{array}			
	\right.		
\end{align*}
The result then follows after examining the various cases.
\end{proof}

\begin{corollary}\label{PropTMuOrders}
If $f(\tau)=q^{\alpha}\tmu(u\tau+v,w\tau+x;\tau)$ is a harmonic weak Maass form,
with $u,v,w,x\in\mathbb{R}$, then the lowest power of 
$q$ appearing in the expansion of the holomorphic part of
$f(\tau)$ is at least $\alpha+\tilde{\nu}(u,w)$, where
\begin{align*}
	\tilde{\nu}(u,w)
	&=
	\frac{1}{2}\left( \Floor{u}-\Floor{w} \right)^2
	+
	\left(\Floor{u}-\Floor{w}\right)\left(\Fractional{u}-\Fractional{w}\right)		
	+
	k( u , w )
	,
\end{align*}
and
\begin{align*}
	k(u, w)
	&=
	\PieceTwo{\nu(\Fractional{u},\Fractional{w}) }
		{\min\left( \frac{1}{8},\nu(\Fractional{u},\Fractional{w}) \right)  }
		{\Fractional{u}-\Fractional{w}\not=\pm\frac{1}{2}}
		{\Fractional{u}-\Fractional{w}=\pm\frac{1}{2}}	
.
\end{align*}
\end{corollary}
\begin{proof}
By (\ref{EqZTheorem1.11P1}) we have
\begin{align*}
	\tmu(u\tau+v,w\tau+x;\tau)
	&=
	(-1)^{\Floor{u}+\Floor{w}}	
	\exp\left(\pi i\tau\left(\Floor{u}-\Floor{v}\right)^2
		+2\pi i\left(\Floor{u}-\Floor{v}\right)\left(\Fractional{u}\tau+v-\Fractional{v}\tau-w\right)  
	\right)
	\\&\quad
	\tmu(\Fractional{u}\tau+v,\Fractional{w}\tau+x;\tau)
.
\end{align*}
However, the lowest power of $q$ appearing in the $q$-expansion of 
$\mu(\Fractional{u}\tau+v,\Fractional{w}\tau)+x;\tau)$ is 
$\nu(\Fractional{u},\Fractional{w})$. Now 
$R((\Fractional{u}-\Fractional{w})\tau+v-x;\tau)$ contributes nothing to the 
holomorphic part unless $\Fractional{u}-\Fractional{v}=\pm\frac{1}{2}$, in
which class it contributes a constant multiple of $q^{\frac{1}{8}}$.
This completes the proof.
\end{proof}


We recall for a modular form $f$ on some congruence subgroup $\Gamma$, 
the invariant order at
$\infty$ is the least power of $q$ appearing in the $q$-expansion at 
$i\infty$. That is, if
\begin{align*}
	f(\tau) &= \sum_{m=m_0}^\infty a(m)\exp(2\pi i \tau m/N )
,
\end{align*}
and $a(m_0)\not=0$, then the invariant order is $m_0/N$. For a modular form,
this is always a finite number.
For a harmonic weak Maass form, we cannot take such an expansion,
however we can do so for the holomorphic part.
If $f$ is a modular form of weight $k$,
$\gcd(\alpha,\gamma)=1$, and 
$A=\TwoTwoMatrix{\alpha}{\beta}{\gamma}{\delta}\in\SLTwo$,
then the invariant order of $f$ at the cusp $\frac{\alpha}{\gamma}$ is the invariant
order at $\infty$ of $(A:\tau)^{-k}f(A\tau)$. In the same fashion, 
if $f$ is a harmonic weak Maass form, then the invariant order of the
holomorphic part of $f$ at the cusp $\frac{\alpha}{\gamma}$ is the  is the invariant
order at $\infty$ of $(A:\tau)^{-k}f(A\tau)$. This value is independent of the choice of
$A$.


\begin{proposition}\label{PropNOrders}
Suppose $a$ and $c$ are integers, $c>0$, and $c\nmid 2a$.
Suppose $\alpha$ and $\gamma$ are non-negative integers and $\gcd(\alpha,\gamma)=1$.
If $\gamma$ is even, then the invariant order of the holomorphic part of 
$N(a,c;\tau)$ at the cusp $\frac{\alpha}{\gamma}$ is at least
\begin{align*}
	-\frac{a^2\gamma^2}{c^2}
	+
	\frac{a\gamma}{c}
	-
	\frac{1}{4}
	+
	2\tilde{\nu}\left( \frac{a\gamma}{c}, \frac{1}{2}  \right)
.
\end{align*}
If $\gamma$ is odd, then the invariant order of $N(a,c;\tau)$ at the cusp 
$\frac{\alpha}{\gamma}$ is at least
\begin{align*}
	-\frac{a^2\gamma^2}{c^2}
	+
	\frac{a\alpha\gamma}{c}
	-
	\frac{\alpha^2}{4}
	+
	\frac{1}{2}\tilde{\nu}\left( \frac{2a\gamma}{c}, \alpha  \right)
.
\end{align*}
\end{proposition}
\begin{proof}
We take $A=\TwoTwoMatrix{\alpha}{\beta}{\gamma}{\delta}\in\SLTwo$.
By Proposition \ref{PropNTransformations},
if $\gamma$ is even, then 
\begin{align*}
	N(a,c;A\tau)
	&=
	\varepsilon
	\exp\left(-\pi i\tau\left(
		\frac{2a^2\gamma^2}{c^2} - \frac{2a\gamma}{c} + \frac{1}{2}		  
	\right)\right)
	\sqrt{\gamma\tau+\delta}
	\,\tilde{\mu}\left( \frac{2a\delta}{c}+\frac{2a\gamma\tau}{c} , \tau ; 2\tau \right)
,
\end{align*}
and if $\gamma$ odd then
\begin{align*}
	N(a,c;A\tau)
	&=
	\varepsilon
	\exp\left(-\pi i\tau\left(
		\frac{2a^2\gamma^2}{c^2} - \frac{2a\alpha\gamma}{c} + \frac{\alpha^2}{2}
	\right)\right)		  
	\sqrt{\gamma\tau+\delta}
	\,\tilde{\mu}\left( 
		\frac{a\delta}{c} + \frac{a\gamma\tau}{c} , 
		\frac{\alpha\tau}{2} + \frac{\beta}{2} ; \frac{\tau}{2} \right)
,
\end{align*}
for some constant $\varepsilon$. The proposition now follows from
Corollary \ref{PropTMuOrders}.
\end{proof}

\begin{proposition}\label{PropMOrders}
Suppose $a$, $c$ and $m$ are integers, $c>0$, $m>0$, $c\nmid a$, and 
$M(a,c;m\tau)$ is a harmonic weak Maass form.
Suppose $\alpha$ and $\gamma$ are non-negative integers and $\gcd(\alpha,\gamma)=1$.
Then the invariant order of the holomorphic part of 
$M(a,c;m\tau)$ at the cusp $\frac{\alpha}{\gamma}$ is at least
\begin{align*}
	-\frac{g^2x^2}{2m}\left(\frac{a}{c}-\frac{1}{2}\right)^2
	+
	\frac{g^2}{m} \tilde{\nu}\left( \frac{ax}{c},\frac{x}{2}  \right)
,
\end{align*}
where $g=\gcd(m,\gamma)$ and $x=m\alpha/g$
.
\end{proposition}
\begin{proof}
We take
$A=\TwoTwoMatrix{\alpha}{\beta}{\gamma}{\delta}\in\SLTwo$.
We set $u=\gamma/g$, so that $x$ and $u$ are relatively prime. We then take
$L=\TwoTwoMatrix{x}{y}{u}{v}\in\SLTwo$.
Next we set
\begin{align*}
	B
	&=
	L^{-1}\TwoTwoMatrix{m\alpha}{m\beta}{\gamma}{\delta}
	=
	\TwoTwoMatrix{g}{m\beta v-\delta y}{0}{m/g}
	.
\end{align*}
The result then follows Corollary \ref{PropTMuOrders} 
as Proposition \ref{PropMTransformation}
gives
\begin{align*}
	M(a,c;mA\tau)
	&=
	M(a,c;L(B\tau))
	\\
	&=	
		\varepsilon
		\sqrt{\gamma\tau+\delta}
		\exp\left( -\pi ix^2\left(\frac{a}{c}-\frac{1}{2}\right)^2 B\tau  \right)
		\tmu\left(\frac{axB\tau+ay}{c}, \frac{xB\tau+y}{2}; B\tau  \right)
,
\end{align*}
for some constant $\varepsilon$.
\end{proof}
While we have only verified that $M(a,c;m\tau)$ is a harmonic weak Maass form 
for $c=7$ and $m=98$, it is clear other values will give harmonic weak Maass forms
as well.

\begin{proposition}\label{PropPOrders}
Suppose $a$ and $c$ are integers, $c>0$, and $c\nmid 2a$.
Suppose $\alpha$ and $\gamma$ are non-negative integers and $\gcd(\alpha,\gamma)=1$.
If $\gamma$ is even, then the invariant order of $P(a,c;\tau)$ at the cusp 
$\frac{\alpha}{\gamma}$ is
\begin{align*}
	\left\{\frac{a\gamma}{c}\right\} - \left\{\frac{a\gamma}{c}\right\}^2
.
\end{align*}
If $\gamma$ is odd, then the invariant order of $P(a,c;\tau)$ at the cusp 
$\frac{\alpha}{\gamma}$ is
\begin{align*}
	\frac{1}{4}\left\{\frac{a\gamma}{c}\right\} 
		- \frac{1}{4}\left\{\frac{a\gamma}{c}\right\}^2
		-\frac{1}{16}
.		
\end{align*}
\end{proposition}
\begin{proof}
We take $A=\TwoTwoMatrix{\alpha}{\beta}{\gamma}{\delta}\in\SLTwo$. If
$\gamma$ is even, then we have
\begin{align*}
	P(a,c;A\tau)
	&=		
		\frac{\varepsilon \eta( {^2}A 2\tau)}
			{\eta(A\tau)^2 \,\, t_{0,\frac{2a}{c}}( ^2A 2\tau) }
	=		
		\frac{\varepsilon
			\sqrt{\gamma\tau+\delta}
			\eta(2\tau)}
		{\eta(\tau)^2 \,\, t_{ \frac{a\gamma}{c}, \frac{2a\delta}{c} }(2\tau) }
	=
		\frac{
			\varepsilon
			\sqrt{\gamma\tau+\delta}
		\eta(2\tau)}{\eta(\tau)^2 \,\, t_{ \Fractional{\frac{a\gamma}{c}}, \frac{2a\delta}{c} }(2\tau) }
,
\end{align*}
where $\varepsilon$ is some constant.
The result then follows as the lowest power of $q$ appearing in the expansion 
of 
\begin{align*}
	\frac{\eta(2\tau)}{\eta(\tau)^2 \,\, t_{ \Fractional{\frac{a\gamma}{c}}, \frac{2a\delta}{c} }(2\tau) }
\end{align*}
is
$
	\left\{\frac{a\gamma}{c}\right\} - \left\{\frac{a\gamma}{c}\right\}^2
$.
The case when $\gamma$ is odd follows by similar calculations, but we instead
choose $A\in\SLTwo$ with $\delta$ even and use that
\begin{align*}
	2A\tau
	&=
	\frac{2a(\tau/2)+\beta }{\gamma(\tau/2) + \delta/2 }	
	.
\end{align*}
We omit the details.
\end{proof}

The following is Lemma 3.2 of \cite{Biagioli}.
\begin{proposition}\label{PropProductOrders}
Suppose $\alpha$ and $\gamma$ are non-negative integers and $\gcd(\alpha,\gamma)=1$.
Then the invariant order of $f_{N,\rho}(\tau)$ at the cusp 
$\frac{\alpha}{\gamma}$ is
\begin{align*}
	\frac{\gcd(N,\gamma)^2}{2N}
	\left(
		\Fractional{\frac{\alpha\rho}{\gcd(N,\gamma)} } 
		-
		\frac{1}{2}
	\right)^2
.
\end{align*}
\end{proposition}

\section{Proof Of Theorem \ref{TheoremMain} }

To begin, we consider the equation
\begin{align}\label{EqAsMaassForms}
	\mathcal{M}(1,7;\tau)
	&=
		\tfrac{A(-16,-8,-8)F_0^{15}}{F_1^3 F_2^3 F_3^3 F_4 F_5 F_6 F_7^2}
		+
		\tfrac{A(-1,1,-1)F_0^{15}}{F_1^3 F_2^3 F_3^3 F_5^3 F_6^2}	
		+
		\tfrac{A(1,-1,-1)F_0^{15}}{F_1^3 F_2^3 F_3^2 F_4^3 F_6^3}
		+
		\tfrac{A(-5,-1,-3)F_0^{15}}{F_1^3 F_2^2 F_3^3 F_4^3 F_5^3}
		+
		\tfrac{A(20,8,12)F_0^{15}}{F_1^3 F_2^3 F_3^2 F_4^2 F_5^3 F_7}
		\nonumber\\&\quad
		+
		\tfrac{A(-3,-5,-1)F_0^{15} F_4}{F_1^3 F_2^3 F_3^2 F_5^2 F_6^3 F_7^2}		
		+
		\tfrac{A(4,4,0)F_0^{15}}{F_1^3 F_2^3 F_4^3 F_5 F_6^3 F_7}		
		+
		\tfrac{A(3,5,1)F_0^{15} F_3^2}{F_1^3 F_2^3 F_4^3 F_5^2 F_6^3 F_7^2}
		+
		\tfrac{A(2,-4,2)F_0^{15}}{F_1^3 F_2^3 F_3^3 F_5^2 F_6^2 F_7}
		+
		\tfrac{A(18,6,10)F_0^{15} F_5}{F_1^3 F_2^3 F_3^2 F_4^3 F_6^3 F_7}
		\nonumber\\&\quad
		+
		\tfrac{A(-12,-2,-8)F_0^{15}}{F_1^3 F_2^3 F_3^2 F_4^2 F_5^2 F_7^2}	
		+
		\tfrac{A(18,6,10)F_0^{15} F_4}{F_1^3 F_2^3 F_3^2 F_5 F_6^3 F_7^3}
		+
		\tfrac{A(-18,-6,-10)F_0^{15} F_3^2}{F_1^3 F_2^3 F_4^3 F_5 F_6^3 F_7^3}
		+
		\tfrac{A(-10,-2,-6)F_0^{15}}{F_1^3 F_2^2 F_3^3 F_4^3 F_5^2 F_7}
		+
		\tfrac{A(-12,-2,-8) F_0^{15}}{F_1^3 F_2^2 F_3 F_4^3 F_5^3 F_7^2}
		\nonumber\\&\quad
		+
		\tfrac{A(12,10,6)F_0^{15}}{F_1^3 F_2^3 F_3^3 F_5 F_6^2 F_7^2}
		+
		\tfrac{A(-30,-6,-18)F_0^{15}}{F_1^3 F_2^3 F_3^2 F_4^2 F_5 F_7^3}
		+
		\tfrac{A(-14,-6,-10)F_0^{15}}{F_1^3 F_2^3 F_3^2 F_4 F_5^3 F_6 F_7}
		+
		\tfrac{A(16,2,10)F_0^{15} F_6}{F_1^3 F_2^3 F_3 F_4^3 F_5^3 F_7^2}
		+
		\tfrac{A(1,-1,1)F_0^{15}}{F_1^3 F_2^3 F_3 F_5^2 F_6^2 F_7^3}		
		\nonumber\\&\quad
		+
		\tfrac{A(37,-1,25)F_0^{15}}{2 F_1^3 F_2^2 F_3^3 F_4^3 F_5 F_7^2}
		+
		\tfrac{A(-1,-3,-1)F_0^{15}}{F_1^3 F_2^2 F_3^2 F_4 F_5^2 F_6^2 F_7^2}
		+
		\tfrac{A(-2,2,-2)F_0^{15}}{F_1^3 F_2 F_3^3 F_4^3 F_6 F_7^3}
		+
		\tfrac{A(1,-1,1)F_0^{3} F_3^3 F_4^2 F_5}{2 F_1^3 F_6^3 F_7^2}
		+
		\tfrac{A(1,-1,1)F_0^{3} F_2 F_3 F_4^3}{F_1^3 F_6 F_7^3}
		\nonumber\\&\quad
		+
		\tfrac{A(-2,2,-2) F_0^{3} F_2^2 F_4^3 F_5^2}{F_1^3 F_3 F_6^2 F_7^3}
		+
		\tfrac{A(8,-8,8)F_0^{15}}{F_1^3 F_2^3 F_3^3 F_6^2 F_7^3}
		+		
		\tfrac{A(-8,8,-8)F_0^{15} F_5^3}{F_1^3 F_2^3 F_3^2 F_4^3 F_6^3 F_7^3}
		+
		\tfrac{A(-40,8,-32)F_0^{15}}{F_1^3 F_2^3 F_3^2 F_4 F_5^2 F_6 F_7^2}
		+
		\tfrac{A(28,-2,20)F_0^{15}}{F_1^3 F_2^3 F_3 F_4^3 F_5 F_6^3}	
		\nonumber\\&\quad
		+
		\tfrac{A(20,6,12)F_0^{15} F_3}{F_1^3 F_2^3 F_4^3 F_5^2 F_6^3 F_7}
		+
		\tfrac{A(8,-8,8)F_0^{15}}{F_1^3 F_2^2 F_3^3 F_4^3 F_7^3}
		+
		\tfrac{A(-8,-6,-4)F_0^{15}}{F_1^3 F_2^2 F_3^2 F_5^3 F_6^3 F_7}
		+
		\tfrac{A(-4,-4,-2)F_0^{15} F_4}{F_1^3 F_2^3 F_3^3 F_5 F_6^3 F_7^2}
		+
		\tfrac{A(2,-2,2)F_0^{15}}{F_1^3 F_2^3 F_3^2 F_5^3 F_6^2 F_7}
		\nonumber\\&\quad
		+
		\tfrac{A(-6,2,-6)F_0^{15}}{F_1^3 F_2^3 F_3 F_4^3 F_6^3 F_7}
		+
		\tfrac{A(-10,-2,-8)F_0^{15} F_4}{F_1^3 F_2^3 F_3 F_5^2 F_6^3 F_7^3}
		+
		\tfrac{A(4,4,2) F_0^{15} F_3}{F_1^3 F_2^3 F_4^3 F_5 F_6^3 F_7^2}
		+
		\tfrac{A(10,2,8) F_0^{15} F_3^3}{F_1^3 F_2^3 F_4^3 F_5^2 F_6^3 F_7^3}
		+
		\tfrac{A(10,2,6)F_0^{15}}{F_1^3 F_2^2 F_3^2 F_4^3 F_5^3 F_7}
		\nonumber\\&\quad
		+	
		\tfrac{A(24,12,16)F_0^{15} F_4}{F_1^3 F_2^3 F_3^3 F_6^3 F_7^3}
		+		
		\tfrac{A(8,8,6)F_0^{15}}{F_1^3 F_2^3 F_3^2 F_5^2 F_6^2 F_7^2}
		+
		\tfrac{A(-8,-8,-6)F_0^{15}}{F_1^3 F_2^3 F_3 F_4^2 F_5^2 F_7^3}
		+
		\tfrac{A(-24,-12,-16)F_0^{15}}{F_1^3 F_2^3 F_5^3 F_6^2 F_7^3}
		+
		\tfrac{A(-24,-8,-18)F_0^{15}}{F_1^3 F_2^2 F_3^3 F_4 F_5^2 F_6^2 F_7}
		\nonumber\\&\quad
		+
		\tfrac{A(0,-4,0)F_0^{15}}{F_1^3 F_2^2 F_3^2 F_4^3 F_5^2 F_7^2}
		+
		\tfrac{A(16,4,10)F_0^{15}}{F_1^3 F_2^2 F_4^3 F_5^3 F_7^3}
		+	
		\tfrac{A(-40,-8,-24)F_0^{15}}{F_1^3 F_2^3 F_3^2 F_5 F_6^2 F_7^3}		
		+
		\tfrac{A(2,4,0)F_0^{15} F_4}{F_1^3 F_2^3 F_3^2 F_5^3 F_6^3 F_7}		
		+
		\tfrac{A(18,4,12)F_0^{15} F_5^2}{F_1^3 F_2^3 F_3 F_4^3 F_6^3 F_7^3}		
		\nonumber\\&\quad
		+
		\tfrac{A(40,8,24)F_0^{15}}{F_1^3 F_2^3 F_3 F_4 F_5^3 F_6 F_7^2}
		+		
		\tfrac{A(-18,-4,-12)F_0^{15}}{F_1^3 F_2^3 F_4^3 F_5^2 F_6^3}
		+
		\tfrac{A(6,-4,4)F_0^{15}}{F_1^3 F_2^2 F_3^3 F_4 F_5 F_6^2 F_7^2}
		+
		\tfrac{A(-18,-4,-12)F_0^{15}}{F_1^3 F_2^2 F_3^2 F_4^3 F_5 F_7^3}
		\nonumber\\&\quad
		+
		\tfrac{A(22,4,12) F_0^{15}}{F_1^3 F_2^2 F_3 F_4 F_5^2 F_6^2 F_7^3}
		+
		\tfrac{A(-8,0-6)F_0^3}{F_4 F_7}
		+
		\tfrac{A(6,4,4)F_0^3}{F_6 F_7}
		+
		\tfrac{A(8,2,4)F_0^3}{F_2 F_7}
		\scriptstyle	
		-
		i A(-8,0-6) N_7(2;\tau)
		-
		i A(6,4,4) N_7(3;\tau)
		\nonumber\\&\quad		
		\scriptstyle	
		-
		i A(8,2,4) N_7(1;\tau)
,
\end{align}
where $F_k=f_{98,7k}(\tau)$ for $1\le k\le 6$  and
$F_0 = \eta(98\tau)$.
By Corollaries
\ref{CorScriptMMaassForm},
\ref{CorN7MaassForm}, and
\ref{CorExtarProductsModularForm}
this is an identity between two harmonic weak Maass forms on
$\Gamma_0(784)\cap\Gamma_1(28)$. We see
(\ref{EqAsMaassForms}) is
the identity in Theorem \ref{TheoremMain}
with the appropriate non-holomorphic parts added to each side.
However using Proposition 
\ref{PropScriptMNonHolomorphicPart},
and rewriting the series with $n\mapsto 7n\pm k$ for $k=1,2,3$,
we find that the non-holomorphic part of the left hand side is equal to
the non-holomorphic part of the right hand side as given by Proposition
\ref{PropN7kNonHolomorphicPart}. Therefore (\ref{EqAsMaassForms}) is equivalent to Theorem 
\ref{TheoremMain}. Furthermore, by subtracting the left hand side
of (\ref{EqAsMaassForms}) from the right hand side, we see
(\ref{EqAsMaassForms}) is equivalent to verifying a weakly holomorphic modular form on
$\Gamma_1(28)\cap\Gamma_0(784)$ is zero. In fact we can work on a larger
group. By Corollary \ref{CorNTransformations1} and Propositions  
\ref{PropPTransformations1},
\ref{PropN7kTransformation1}, and
\ref{PropExtraProductsTransformation1}
each individual term in (\ref{EqAsMaassForms}) satisfies
\begin{align*}
	f(A\tau) 
	&= 
	\nu( ^2A)^{-3} (-1)^{\beta+\frac{\alpha-1}{2}} i^{-\alpha\beta} f(\tau)
,
\end{align*}
for $A\in\Gamma_0(98)\cap\Gamma_1(14)$.
We subtract $\mathcal{M}(1,7;\tau)$ from both sides
and divide by $\frac{F_0^{15}}{F_1^3 F_2^3 F_3^3 F_4 F_5 F_6 F_7^2}$ 
to obtain the equivalent identity
\begin{align}\label{EqAsModularForms}
	0 
	&=
	\scriptstyle	
	A(-16,-8,-8)
	+\tfrac{A(-1,1,-1)F_4 F_7^2 }{F_5^2 F_6 }
	+\tfrac{A(1,-1,-1)F_3 F_5 F_7^2 }{F_4^2 F_6^2 }
	+\tfrac{A(-5,-1,-3)F_2 F_6 F_7^2 }{F_4^2 F_5^2 }
	+\tfrac{A(20,8,12)F_3 F_6 F_7 }{F_4 F_5^2 }
	\nonumber\\&\quad
	+\tfrac{A(-3,-5,-1)F_3 F_4^2 }{F_5 F_6^2 }
	+\tfrac{A(4,4,0)F_3^3 F_7 }{F_4^2 F_6^2 }
	+\tfrac{A(3,5,1)F_3^5 }{F_4^2 F_5 F_6^2 }
	+\tfrac{A(2,-4,2)F_4 F_7 }{F_5 F_6 }
	+\tfrac{A(18,6,10)F_3 F_5^2 F_7 }{F_4^2 F_6^2 }
	+\tfrac{A(-12,-2,-8)F_3 F_6 }{F_4 F_5 }
	\nonumber\\&\quad
	+\tfrac{A(18,6,10)F_3 F_4^2 }{F_6^2 F_7 }
	+\tfrac{A(-18,-6,-10)F_3^5 }{F_4^2 F_6^2 F_7 }
	+\tfrac{A(-10,-2,-6)F_2 F_6 F_7 }{F_4^2 F_5 }
	+\tfrac{A(-12,-2,-8)F_2 F_3^2 F_6 }{F_4^2 F_5^2 }
	+\tfrac{A(12,10,6)F_4 }{F_6 }
	\nonumber\\&\quad
	+\tfrac{A(-30,-6,-18)F_3 F_6 }{F_4 F_7 }
	+\tfrac{A(-14,-6,-10)F_3 F_7 }{F_5^2 }
	+\tfrac{A(16,2,10)F_3^2 F_6^2 }{F_4^2 F_5^2 }
	+\tfrac{A(1,-1,1)F_3^2 F_4 }{F_5 F_6 F_7 }
	+\tfrac{A(37,-1,25)F_2 F_6 }{2F_4^2 }
	\nonumber\\&\quad
	+\tfrac{A(-1,-3,-1)F_2 F_3 }{F_5 F_6 }
	+\tfrac{A(-2,2,-2)F_2^2 F_5 }{F_4^2 F_7 }
	+\tfrac{A(1,-1,1)F_2^3 F_3^6 F_4^3 F_5^2 }{2F_0^12 F_6^2 }
	+\tfrac{A(1,-1,1)F_2^4 F_3^4 F_4^4 F_5 }{F_0^12 F_7 }
	+\tfrac{A(-2,2,-2)F_2^5 F_3^2 F_4^4 F_5^3 }{F_0^12 F_6 F_7 }
	\nonumber\\&\quad	
	+\tfrac{A(8,-8,8)F_4 F_5 }{F_6 F_7 }
	+\tfrac{A(-8,8,-8)F_3 F_5^4 }{F_4^2 F_6^2 F_7 }
	+\tfrac{A(-40,8,-32)F_3 }{F_5 }
	+\tfrac{A(28,-2,20)F_3^2 F_7^2 }{F_4^2 F_6^2 }
	+\tfrac{A(20,6,12)F_3^4 F_7 }{F_4^2 F_5 F_6^2 }
	+\tfrac{A(8,-8,8)F_2 F_5 F_6 }{F_4^2 F_7 }
	\nonumber\\&\quad	
	+\tfrac{A(-8,-6,-4)F_2 F_3 F_4 F_7 }{F_5^2 F_6^2 }
	+\tfrac{A(-4,-4,-2)F_4^2 }{F_6^2 }
	+\tfrac{A(2,-2,2)F_3 F_4 F_7 }{F_5^2 F_6 }
	+\tfrac{A(-6,2,-6)F_3^2 F_5 F_7 }{F_4^2 F_6^2 }
	+\tfrac{A(-10,-2,-8)F_3^2 F_4^2 }{F_5 F_6^2 F_7 }
	\nonumber\\&\quad	
	+\tfrac{A(4,4,2)F_3^4 }{F_4^2 F_6^2 }
	+\tfrac{A(10,2,8)F_3^6 }{F_4^2 F_5 F_6^2 F_7 }
	+\tfrac{A(10,2,6)F_2 F_3 F_6 F_7 }{F_4^2 F_5^2 }
	+\tfrac{A(24,12,16)F_4^2 F_5 }{F_6^2 F_7 }
	+\tfrac{A(8,8,6)F_3 F_4 }{F_5 F_6 }
	+\tfrac{A(-8,-8,-6)F_3^2 F_6 }{F_4 F_5 F_7 }
	\nonumber\\&\quad
	+\tfrac{A(-24,-12,-16)F_3^3 F_4 }{F_5^2 F_6 F_7 }
	+\tfrac{A(-24,-8,-18)F_2 F_7 }{F_5 F_6 }
	+\tfrac{A(0,-4,0)F_2 F_3 F_6 }{F_4^2 F_5 }
	+\tfrac{A(16,4,10)F_2 F_3^3 F_6 }{F_4^2 F_5^2 F_7 }
	+\tfrac{A(-40,-8,-24)F_3 F_4 }{F_6 F_7 }
	\nonumber\\&\quad
	+\tfrac{A(2,4,0)F_3 F_4^2 F_7 }{F_5^2 F_6^2 }
	+\tfrac{A(18,4,12)F_3^2 F_5^3 }{F_4^2 F_6^2 F_7 }
	+\tfrac{A(40,8,24)F_3^2 }{F_5^2 }
	+\tfrac{A(-18,-4,-12)F_3^3 F_7^2 }{F_4^2 F_5 F_6^2 }
	+\tfrac{A(6,-4,4)F_2 }{F_6 }
	+\tfrac{A(-18,-4,-12)F_2 F_3 F_6 }{F_4^2 F_7 }
	\nonumber\\&\quad	
	+\tfrac{A(22,4,12)F_2 F_3^2 }{F_5 F_6 F_7 }
	%
	+
	\tfrac{A(-8,0-6) F_1^3 F_2^3 F_3^3 F_5 F_6 }{F_0^{12} F_7}
	+
	\tfrac{A(6,4,4) F_1^3 F_2^3 F_3^3 F_4 F_5 }{F_0^{12} F_7}
	+
	\tfrac{A(8,2,4) F_1^3 F_2^2 F_3^3 F_4 F_5 F_6 }{F_0^{12} F_7}
	\nonumber\\&\quad
	\scriptstyle	
	-
	\tfrac{i A(-8,0-6) F_1^3 F_2^3 F_3^3 F_4 F_5 F_6 F_7^2  N_7(2;\tau) }{F_0^{15}} 
	-
	\tfrac{i A(6,4,4) F_1^3 F_2^3 F_3^3 F_4 F_5 F_6 F_7^2  N_7(3;\tau) }{F_0^{15}}
	-	
	\tfrac{iA(8,2,4)  F_1^3 F_2^3 F_3^3 F_4 F_5 F_6 F_7^2 N_7(1;\tau)  }{F_0^{15}}
	\nonumber\\&\quad
	-
	\tfrac{F_1^3 F_2^3 F_3^3 F_4 F_5 F_6 F_7^2 \mathcal{M}(1,7;\tau)}{F_0^{15}}		
.
\end{align}
We let $RHS$ denote the right hand side of (\ref{EqAsModularForms})
and let $\Gamma=\Gamma_0(98)\cap\Gamma_1(14)$. We know
$RHS$ is a modular function on $\Gamma$.

The valence formula, for modular functions can be stated as follows.
Suppose $f$ is a modular function on some congruence subgroup 
$\Gamma\subset\SLTwo$.
Suppose $A=\TwoTwoMatrix{\alpha}{\beta}{\gamma}{\delta}\in\SLTwo$,
we then have a cusp $\zeta=A(\infty)=\frac{\alpha}{\gamma}$.
We let $ord(f;\zeta)$ denote the invariant order of $f$ at
$\zeta$. We define the width of $\zeta$ with respect to $\Gamma$
as $width_\Gamma(\zeta):=w$, where $w$ is the least positive integer such that
$A\TwoTwoMatrix{1}{w}{0}{1}A^{-1}\in\Gamma$. We then define
the order of $f$ at $\zeta$ with respect to $\Gamma$ as
$ORD_\Gamma(f;\zeta)=ord(f;\zeta) width_\Gamma(\zeta)$.
For $z\in\mathcal{H}$ we let $ord(f;z)$ denote the order of $f$ at $z$ as a
meromorphic function. We then define the order of $f$ at $z$ with respect to
$\Gamma$ as $ORD_\Gamma(f;z)=ord(f;z)/m$ where $m$ is the order of $z$ as a
fixed point of $\Gamma$ (so $m=1$, $2$, or $3$).
If $f$ is not the zero function 
and $\mathcal{D}\subset\mathcal{H}\cup\mathbb{Q}\cup\{i\infty\}$ is a
fundamental domain for the action of $\Gamma$ on $\mathcal{H}$ along 
with a complete set of inequivalent cusps for the action,
then
\begin{align*}
	\sum_{\zeta \in \mathcal{D}} ORD_\Gamma(f;\zeta) = 0
.
\end{align*}

A complete set of inequivalent cusps, along with their widths, 
for $\Gamma_0(98)\cap\Gamma_1(14)$ is
$$
\renewcommand{\arraystretch}{1.1}	
	\begin{array}{l|ccccccccccccccccccccc}
		\mbox{cusp}&
			0& \frac{1}{14}& \frac{3}{38}& \frac{2}{25}& \frac{1}{12}& 
			\frac{3}{35}& \frac{2}{23}& \frac{3}{28}& \frac{5}{42}& \frac{1}{8}& 
			\frac{1}{7}& \frac{5}{28}& \frac{3}{14}& \frac{8}{35}& \frac{5}{21}& 
			\frac{2}{7}& \frac{17}{56}& \frac{11}{35}& \frac{9}{28}& \frac{5}{14}& 
			\frac{18}{49} 
		\\		
		\mbox{width}&				
			98& 1& 49& 98& 49& 
			2& 98& 1& 1& 49& 
			2& 1& 1& 2& 2& 
			2& 1& 2& 1& 1& 
			2  
		\\			
		\hline
		\mbox{cusp}&
			\frac{37}{98}& \frac{8}{21}& \frac{19}{49}& \frac{11}{28}& \frac{3}{7}& 
			\frac{22}{49}& \frac{16}{35}& \frac{45}{98}& \frac{13}{28}& \frac{10}{21}&		
			\frac{27}{56}& \frac{29}{56}& \frac{11}{21}& \frac{15}{28}& \frac{19}{35}& 
			\frac{4}{7}& \frac{13}{21}& \frac{9}{14}& \frac{39}{56}& \frac{5}{7}& 
			\frac{16}{21} 
		\\
		\mbox{width}&
			1& 2& 2& 1& 2& 
			2& 2& 1	& 1& 2& 
			1& 1& 2& 1& 2& 
			2& 2& 1& 1& 2& 
			2 
		\\
		\hline			
		\mbox{cusp}&
			\frac{27}{35}& \frac{11}{14}& \frac{6}{7}& \frac{37}{42}& \frac{13}{14}&
			\infty
			&&&&&&&&&&&&&&&
		\\
		\mbox{width}&
			2& 1& 2& 1& 1& 
			1
			&&&&&&&&&&&&&&&
	\end{array}
$$
We let $\mathcal{D}$ denote these cusps along with a fundamental region of the 
action of $\Gamma$.

We note $RHS$ has no poles on $\mathcal{H}$, but it may have zeros on $\mathcal{H}$. 
We take a lower bound on the
orders at the non-infinite cusps by taking the minimum order of each of the 
individual summands in (\ref{EqAsModularForms}), which we compute with
Propositions \ref{PropNOrders},
\ref{PropMOrders},
\ref{PropPOrders}, and
\ref{PropProductOrders}.
This lower bound yields
\begin{align*}
	\sum_{\zeta\in\mathcal{D}} ORD_{\Gamma}(RHS;\zeta)
	\ge	
	ord(RHS,\infty)
	-109.
\end{align*}
However, we can expand $RHS$ as a $q$-series 
and find the coefficients of $RHS$ are zero to at least 
$q^{110}$. Thus
\begin{align*}
	\sum_{\zeta\in\mathcal{D}} ORD_{\Gamma}(RHS;\zeta)
	\ge	
	1,
\end{align*}
and so $RHS$ must be identically zero by the valence formula. 
This proves (\ref{EqAsModularForms}), which is equivalent to
(\ref{EqAsMaassForms}), which is equivalent to Theorem \ref{TheoremMain}.
Thus Theorem \ref{TheoremMain} holds.

It is very fortunate that each of the $7$ terms of the $7$-dissection of 
$\mathcal{M}(1,7)$ transforms the same as $\mathcal{M}(1,7)$ under
$\Gamma_0(98)\cap\Gamma_1(14)$. In general the
terms of a dissection are harmonic Maass forms, but possibly with respect
to a smaller subgroup than $\mathcal{M}(a,c)$.

\bibliographystyle{abbrv}
\bibliography{overpartitionRankDifferencesMod7ByMaassFormsRef}

\end{document}